\documentclass[]{amsart}

\usepackage[utf8x]{inputenc}
\usepackage{lmodern}
\usepackage[T1]{fontenc}

\usepackage{mathrsfs}
\usepackage{amsmath,amssymb,amsthm}

\usepackage[pdftex,
			pdfauthor={James Maunder},
            pdftitle={Homotopy transfer for L-infinity structures and the BV-formalism},
            pdfsubject={},
            pdfkeywords={},
            pdfproducer={LaTeX with hyperref},
            pdfcreator={PDFLaTeX},
			pdfencoding=auto,
			psdextra
			]{hyperref}

\usepackage{tikz}
\usetikzlibrary{trees,matrix}

\usepackage[framestyle=fbox,
			framefit=yes,
			heightadjust=all,
			framearound=all
			]{floatrow}

\newtheorem{theorem}{Theorem}[section]

\newtheorem{proposition}[theorem]{Proposition}
\newtheorem{definition}[theorem]{Definition}
\newtheorem{lemma}[theorem]{Lemma}

\theoremstyle{definition}

\newtheorem{remark}[theorem]{Remark}

\newcommand{\Linf}{L$_\infty$}
\newcommand{\BVinf}{BV$_{\!\infty}$}
\newcommand{\MC}{\operatorname{MC}}

\newcommand{\Aut}{\operatorname{Aut}}
\newcommand{\Der}{\operatorname{Der}}
\newcommand{\im}{\operatorname{im}}
\newcommand{\homology}{\operatorname{H}}
\newcommand{\Dev}{\operatorname{D_{ev}}}
\newcommand{\Dod}{\operatorname{D_{od}}}
\newcommand{\Dodq}{\operatorname{D_{od,q}}}
\newcommand{\Hev}{\operatorname{H_{ev}}}
\newcommand{\Hod}{\operatorname{H_{od}}}
\newcommand{\Hodq}{\operatorname{H_{od,q}}}
\newcommand{\id}{\operatorname{id}}

\allowdisplaybreaks

\title{Homotopy transfer for L-infinity structures and the BV-formalism}

\author{James Maunder}

\address{School of Engineering, Computing, and Mathematics\\
	Oxford Brookes University\\
	Oxford OX33 1HX\\ UK}
	
\email{jmaunder@brookes.ac.uk}

\begin{document}

\begin{abstract}
Explicit constructions for the minimal models of general and unimodular \Linf-algebra structures are given using the BV-formalism of mathematical physics and the perturbative expansions of integrals. In particular, the general formulas for the minimal model of an \Linf-algebra structure are an instance of the Homotopy Transfer Theorem and we recover the known formulas of the structure in terms of sums over rooted trees discussing their relation to Feynman diagrams.
\end{abstract}

\maketitle

\section*{Introduction}

The Homotopy Transfer Theorem (HTT) describes how to use the data of two quasi-isomorphic chain complexes with an algebraic structure on one of them to construct an algebraic structure on the other one in such a way that these two structures are quasi-isomorphic. Of course, the notion of quasi-isomorphism here needs to be in the homotopy algebra sense, as the transferred structure is usually subject to some coherent system of higher homotopies. One reason that the HTT is so important is that it applies to algebras over many operads and results in explicit constructions, see \cite{loday_vallette,vallette_alg_htpy_operad} for more details. The HTT, therefore, unifies many objects appearing in algebra, geometry, topology, and mathematical physics. In this paper, the perspective taken allows the unification of minimal models of \Linf-algebra structures with Feynman diagrams appearing in the path integrals of the BV-formalism, similar to \cite{chuang_laz_feynman} where minimal models are explicitly constructed for algebras over the cobar-construction of a differential graded modular operad.

Building upon earlier work of Zinn-Justin \cite{zinn-justin}, Kallosh \cite{kallosh}, and deWit and van Holten \cite{dewit_vanholten}, the BV-formalism was originally introduced in mathematical physics as a tool to quantise gauge theories \cite{batalin_vilkovisky_gauge_algebra_quantization, batalin_vilkovisky_quantization_of_gauge_algebras}. The BV-formalism has since been utilised in several other fields of mathematics such as deformation quantisation \cite{cattaneo_felder_deformation}, manifold invariants \cite{cattaneo_mnev}, and link invariants \cite{iacovino}. This success could be attributed to the fruitfulness of the framework of the BV-formalism --- odd symplectic geometry, homological algebra, and path integrals can all be employed in a cohesive and powerful manner. In particular, the BV-formalism is adept at constructing invariants. These invariants may be the homotopy classes of algebraic structures, see for example the Homological Perturbation Lemma's appearance in \cite{doubek_jurco_pulmann, gwilliam_thesis} or unpublished work of Carlo Albert \cite{albert}.

The BRST symmetry \cite{BRS_higgs-kibble, BRS_renormalisation_of_gauge, BRS_renormalisation_of_higgs-kibble, tyutin} is central to the BV-formalism. Additional details and references can be found in \cite{henneaux, henneaux_titelboim, fuster_henneaux_maas, barnich_brandt_henneaux} for example.

Before proceeding, it should be stated that Khudaverdian is responsible for the initial geometric formulation of the BV-formalism \cite{khudaverdian_semidensities,khudaverdian_geometry_of_superspace,khudaverdian_BV_and_odd_symp_geo,khudaverdian_nersessian}. Schwarz also provided a more modern geometric formulation \cite{schwarz}. It should be stated that there exist many papers which consider the BV-formalism from different viewpoints, examples include \cite{costello,getzler_BV,kosmann-schwazbach_monterde,severa}.

The study of \Linf-algebras in the language of formal geometry is well-aligned with the formal geometry of the BV-formalism. Combining this observation with the earlier observation of transfer of algebraic structure gives the viewpoint of the current paper and \cite{braun_maunder}.

One final point to make is that the constructions herein only make partial use of the `quantum' structure apparent in the BV-formalism. Indeed, the transferred \Linf-algebra structures are at the `tree level truncation' (achieved by setting $\hbar=0$) and so all the higher genus structure is forgotten. This is a point that the author intends to revisit in future work.

The paper is organised as follows. Section~\ref{sec_linf_structres} recalls the construction of \Linf-algebra structures as solutions to the Maurer-Cartan equation in certain differential graded Lie algebras and discusses the notions of equivalence for these structures, including the statement of a theorem of Schlessinger and Stasheff (Theorem~\ref{thm_schlessinger_stasheff}) given originally in \cite{schlessinger_stasheff_deformation_rational}. Unimodular and quantum variants of \Linf-algebra structures and their notions of equivalence are also recalled in this section. Section~\ref{sec_doubling} recalls the doubling constructions, given in \cite{braun_laz_unimodular}, and describes how these constructions transfer \Linf-algebra structures and equivalence of such structures. Section~\ref{sec_halving} discusses the one-sided inverses of the aforementioned doubling constructions and when these constructions are compatible with algebraic structures and their equivalences. Section~\ref{sec_BV} serves as a brief resource for the relevant parts of the BV-formalism. In particular, this section establishes this paper's notation of strong deformation retracts, recalls the main results of \cite{braun_maunder} (Theorem~\ref{thm_braun_maunder}), and discusses path integrals in the particular case of `doubled' constructions. Section~\ref{sec_minimal_unimodular} and Section~\ref{sec_minimal_linf} describe how the constructions of the preceding parts of the paper can be employed to construct the minimal models of unimodular \Linf-algebras and \Linf-algebras, respectively. Finally, Section~\ref{sec_HTT} recasts the final theorem of Section~\ref{sec_minimal_linf} (Theorem~\ref{thm_minimal_model_for_general}) into the Homotopy Transfer Theorem (Theorem~\ref{thm_HTT_via_multi-linear}) by describing how the sums over Feynman diagrams transform into the familiar sum over rooted trees.

\section*{Notation and conventions}

Fix the real numbers, $\mathbb{R}$, as the base field. All unmarked tensors are assumed to be over $\mathbb{R}$, unless otherwise stated. We will work in the category of `super vector spaces' that is the category of differential $\mathbb{Z}/2\mathbb{Z}$-graded real vector spaces. Some of the results of this paper can be made to make sense in the $\mathbb{Z}$-graded case with suitable adaptations. However, given the techniques used in this paper, it makes most sense to work in the $\mathbb{Z}/2\mathbb{Z}$-graded setting. Naturally, one can consider a $\mathbb{Z}$-graded object as a $\mathbb{Z}/2\mathbb{Z}$-graded object by taking the grading modulo $2$.

The degree (or parity) of a homogeneous element $v$ of some super vector space will be denoted $|v|$. It is well established notation to refer to those elements of homogeneous degree $0$ as `even' and those of homogeneous degree $1$ as `odd'. We will adhere to this notation herein. Accordingly, the dimension of a super vector space is given by a pair $(m|n)$ where $m$ is the dimension (in the non-graded sense) of the of the subspace of even elements and $n$ is the dimension (in the non-graded sense) of the subspace of odd elements. The total dimension of a super vector space of dimension $(m|n)$ is given by the sum $m+n$. A super vector space is said to be finite-dimensional if, and only if, its total dimension is finite. The category of super vector spaces also has an internal hom-functor, making it a closed symmetric monoidal category.

Let the super vector space of dimension $(0|1)$ be denoted by $\Pi \mathbb{R}$. Given some super vector space $V$, let $\Pi V:=V\otimes \Pi \mathbb{R}$. This construction defines an endo-functor $\Pi$ called parity reversion.

Any algebra is always the appropriate notion in the category of super vector spaces. The expressions `differential (super)graded', `commutative differential graded algebra', and `differential graded Lie algebra' are abbreviated to `dg', `cdga', and `dgla', respectively.

A pseudo-compact super vector space is one given by a projective limit of super vector spaces of finite total dimension. Taking the inverse limit endows a pseudo-compact super vector space with a topology. Therefore, all linear maps of pseudo-compact super vector spaces are assumed to be continuous. In particular, the dual of a pseudo-compact super vector space is the topological dual. The algebraic and topological duals will both be denoted by a superscript asterisk. Using the correct notion of dual has the luxury of always having $(V^*)^*\cong V$ without any finiteness conditions. Similarly, it will always be the case that $(V \otimes V)^*\cong V^* \otimes V^*$ since the tensor product of two pseudo-compact super vector spaces will always be the completed tensor product. More details on pseudo-compact objects can be found in the literature \cite{gabriel,keller_yang,vandenbergh}. It is important to note that the functor $V\mapsto V^*$ establishes a symmetric monoidal equivalence between the category of pseudo-compact super vector spaces and the opposite category of super vector spaces. Meaning, a pseudo-compact dg algebra is equivalently a dg coalgebra, for instance. The main example of a pseudo-compact super vector space within this text is the completed symmetric algebra of a finite-dimensional super vector space $V$. In fact, this is an example of a pseudo-compact cdga (it is the dual of the cofree dg cocommutative coalgebra on $V$). Explicitly, the completed symmetric algebra is the pseudo-compact algebra $\hat{S}(V):=\prod_{i=0}^\infty S^i(V)$, that is as the direct product of symmetric tensor powers of the super vector space $V$. The symmetric algebra $S(V)$ is a subalgebra of $\hat{S}(V)$. 

A pronilpotent dgla (or cdga) is one given by a projective limit of nilpotent dglas (or cdgas). Nilpotent here is in the sense of `global' nilpotence: the descending central series stabilises at zero.

Let $\mathbb{R}[|\hbar|]$ denote the algebra of formal power series in the formal parameter $\hbar$ with real coefficients and $|\hbar|=2$. Given a super vector space $V$, let $V[|\hbar|]:= V\otimes \mathbb{R}[|\hbar|]$ be the $\mathbb{R} [|\hbar|]$-module of formal power series in $\hbar$ with values in $V$.

\section{Various L-infinity algebra structures}\label{sec_linf_structres}

From the formal geometric point of view, \Linf-algebra structures are given by elements of certain dglas satisfying the Maurer-Cartan (MC) equation. We will first establish our notation for MC elements before reviewing the definition of structures for \Linf-algebras, unimodular \Linf-algebras, and quantum \Linf-algebras, as well as the dglas which these structures belong to.

\subsection{Maurer-Cartan elements}

The purpose of this section is to establish the notation of the paper; it is very brief and contains nothing that will be unknown to an expert.

Let $\mathfrak{g}$ be a dgla with differential denoted by $d$ and bracket denoted by $[,]$. An element $\xi\in\mathfrak{g}$ is said to satisfy the MC equation if
\[
d(\xi) + \frac{1}{2}[\xi,\xi]=0.
\]
The set of all elements in $\mathfrak{g}$ satisfying the MC equation will be denoted by $\MC(\mathfrak{g})$.

\begin{remark}
In general, an element must be odd to satisfy the MC equation. In the case of an odd dgla, however, an element must be even to satisfy the MC equation.
\end{remark}

It is well-known that a morphism of dglas maps MC elements to MC elements and so MC defines a functor from the category of dglas to the category of sets.

Let $k[t,dt]$ denote the free unital cdga on the generators $t$ and $dt$ of degrees $0$ and $1$, respectively, such that $d(t)=dt$. For a pronilpotent dgla $\mathfrak{g}$, let $\mathfrak{g}[t,dt]:=\mathfrak{g}\hat{\otimes} k[t,dt]$.

Two MC elements $\xi,\nu \in \MC (\mathfrak{g})$ are said to be homotopic if there exists an element $h(t)\in \MC (\mathfrak{g}[t,dt])$ such that $h(0)=\xi$ and $h(1)=\nu$. The set of equivalence classes of MC elements in $\mathfrak{g}$ under the homotopy relation will be referred to as the Maurer-Cartan moduli space.

\begin{theorem}[Schlessinger and Stasheff]\label{thm_schlessinger_stasheff}
Two MC elements in a pronilpotent dgla are homotopic if, and only if, they are gauge equivalent.
\end{theorem}

For a proof of the above theorem, see either the original paper \cite{schlessinger_stasheff_deformation_rational} or \cite{chuang_laz_feynman}.

\subsection{L-infinity algebra structures}

Fix some super vector space $V$. Let us denote by $\Der (\hat{S}V^\ast )$ the space of derivations of $\hat{S}V^\ast$. Given the formal geometric setting we work in, a derivation may also be referred to as a vector field.

Equipped with the commutator bracket (denoted $[,]$), $\Der (\hat{S}V^\ast )$ becomes a graded Lie algebra. Moreover, the differential $d$ on $V$ induces a differential on $\Der (\hat{S}V^\ast )$ (via the Lie derivative) making it a dgla. Despite the abuse of notation, the induced differential on $\Der (\hat{S}V^\ast )$ will also be denoted by $d$.

Choosing a basis $\lbrace x_i \rbrace_{i\in I}$ in $V^\ast$, we may write any derivation $\xi \in \Der (\hat{S}V^\ast )$ as
\[
\xi = \sum_{i\in I} f_i^0 \partial_{x_i} + \sum_{i\in I} f_i^1 \partial_{x_i} + \cdots + \sum_{i\in I} f_i^k \partial_{x_i} + \cdots,
\]
where $f_i^k$ is a linear combination of monomials of order $k$ in the $x_i$s. If $\xi = \sum_{i\in I} f_i^n \partial_{x_i}$ for some fixed value $n$, then $\xi$ is said be a vector field of order $n$. It is clear that the order of a vector field is independent of the choice of basis. The space of all vector fields of order $\geq n$ will be denoted by $\Der_{\geq n} (\hat{S}V^\ast )$.

\begin{definition}\label{def_L-inf}
Let $V$ be a super vector space. An \Linf-algebra structure on $V$ is an element
\[
m\in \MC (\Der_{\geq 2} (\hat{S}\Pi V^\ast )).
\]
The pair $(V,m)$ will be referred to as an \Linf-algebra and the algebra $\hat{S}\Pi V^\ast$ with the differential $d+m$ will be referred to as its representing cdga.
\end{definition}

\begin{remark}
Letting $V$ have a vanishing differential in Definition~\ref{def_L-inf}, an \Linf-algebra structure on $V$ can equivalently be defined as a MC element belonging to the dgla $\Der_{\geq 1} (\hat{S}\Pi V^\ast )$. However, this latter dgla is not pronilpotent and so it is somewhat less convenient than using the pronilpotent dgla $\Der_{\geq 2} \left(\hat{S}\Pi V^\ast \right)$. For instance, Theorem~\ref{thm_schlessinger_stasheff} no longer applies when the dgla is not pronilpotent.
\end{remark}

\begin{definition}
Let $(V,m_V)$ and $(W,m_W)$ be two \Linf-algebras. An \Linf-morphism
\[
f\colon (V,m_V) \to (W,m_W)
\]
is an cdga morphism between their representing cdgas $f\colon \hat{S}\Pi W^\ast \to \hat{S}\Pi V^\ast$.
\end{definition}

\Linf-algebra structures and morphisms are traditionally defined using multi-linear maps. The viewpoint given above is equivalent to the traditional one. Indeed, any derivation $m\in \Der_{\geq 2} (\hat{S}\Pi V^\ast )$ can be written as
\[
m = m_2 + m_3 + \cdots,
\]
where $m_n$ is a derivation of order $n$. Therefore, $m$ is determined by the collection of maps
\[
m_n \colon \Pi V^\ast \to \left(\left(\Pi V^\ast \right)^{\otimes n} \right)_{S_n}.
\]
Moreover, there is an identification of $S_n$ coinvariants and $S_n$ invariants:
\[
i_n \colon \left(\left(\Pi V^\ast\right)^{\otimes n} \right)_{S_n} \to \left(\left(\Pi V^\ast\right)^{\otimes n} \right)^{S_n} \cong \left(\left(\Pi V^{\otimes n} \right)_{S_n} \right)^\ast,
\]
given by $i_n (x_1 \otimes \cdots \otimes x_n ) = \sum_{\sigma \in S_n} \sigma (x_1 \otimes \cdots \otimes x_k)$. Taking the dual of the composite $i_n\circ m_n$ defines a map
\[
\widetilde{m}_n \colon \left(\Pi V^{\otimes n}\right)_{S_n} \to \Pi V.
\]
The Maurer-Cartan equation for $m$ is equivalent to the appropriate conditions on the symmetric multi-linear maps $\widetilde{m}_n$; for example, $\widetilde{m}_2$ will satisfy the Jacobi identity up to a homotopy (with exact term depending upon $\widetilde{m}_3$).

Using a similar argument, one can show that an \Linf-morphism
\[
f\colon \hat{S}\Pi W^\ast \to \hat{S}\Pi V^\ast
\]
is equivalent to a collection of symmetric multi-linear maps
\[
\widetilde{f}_n \colon \left(\Pi V\right)^{\otimes n} \to \Pi W
\]
for $n=1,2,3,\cdots$ of even parity satisfying the appropriate conditions.

It is generally more convenient to think of \Linf-algebras and morphisms geometrically, rather than via multi-linear maps. However, the notions of \Linf-isomorphisms and \Linf-quasi-isomorphisms are more easily defined from the viewpoint of multi-linear maps. Moreover, the statement of Theorem~\ref{thm_HTT_via_multi-linear} discusses \Linf-algebra structures using multi-linear maps.

\begin{definition}
An \Linf-morphism $f\colon  \hat{S}\Pi W^\ast \to \hat{S}\Pi V^\ast$ is said to be an \Linf-(quasi)-isomorphism if its linear component $\tilde{f}_1 \colon \Pi V \to \Pi W$ is a (quasi)-isomorphism. If $V=W$ and $\widetilde{f}_1$ is the identity map of $V$, then $f$ is called a pointed \Linf-morphism.
\end{definition}

Note that a pointed \Linf-morphism is necessarily an \Linf-isomorphism. Moreover, the notion of pointed \Linf-isomorphism is equivalent to that of gauge equivalence of MC elements (see the appendix of \cite{braun_laz_unimodular}, for example).

\begin{definition}
Two \Linf-algebra structures supported on the same super vector space are said to be equivalent when the following equivalent conditions are satisfied:
\begin{itemize}
\item they are pointed \Linf-isomorphic;
\item they are homotopic;
\item they are gauge equivalent.
\end{itemize}
\end{definition}

\begin{remark}
Here, we are considering equivalence of \Linf-algebra structures supported on the same space. It is possible to consider equivalence of \Linf-algebra structures on different spaces using \Linf-quasi-isomorphisms. However, we wish primarily to make use of homotopy, forcing us to work with structures supported on the same space. Further, working with equivalence of structures supported on the same space is necessary for unimodular and quantum \Linf-algebra structures, because there is no useful notion of unimodular or quantum \Linf-morphism.
\end{remark}

\subsection{Unimodular L-infinity algebra structures}

Unimodular \Linf-algebra structures are defined as MC elements in a certain dgla. We will now define this dgla and recall some of the basic properties of unimodular \Linf-algebras. For more details see \cite{braun_laz_unimodular,granaker}.

\begin{definition}
Let $V$ be some super vector space. Denote by $\mathfrak{g}[V]$ the graded Lie algebra given by the semi-direct product
\[
\Der_{\geq 2} \left(\hat{S} V^\ast \right) \ltimes \Pi \hat{S}_{\geq 1} V^\ast,
\]
with bracket given by
\[
\left[(\xi,\Pi f),(\nu,\Pi g)\right] = \left( [\xi,\nu], \Pi \xi(g) + (-1)^{(|f|+1)\nu} \Pi \nu (f) \right),
\]
where $\xi,\nu\in \Der (\hat{S}V^\ast)$ and $f,g\in \hat{S}V^\ast$.
\end{definition}

In order to introduce a differential on $\mathfrak{g}[V]$ we must recall the definition of the divergence of a vector field.

\begin{definition}
Let $V$ be a super vector space with coordinates $\lbrace x_i \rbrace_{i\in I}$. The divergence of a vector field $\xi= f_i \partial_{x_i} \in \Der (\hat{S}V^\ast )$ is given by
\[
\nabla (\xi) = (-1)^{|f_i||x_i|} \partial_{x_i}f_i \in \hat{S}V^\ast.
\]
The general definition of divergence is given by extending linearly.
\end{definition}

\begin{proposition}
Let $V$ be a super vector space. Let $d_e\colon \mathfrak{g}[V] \to \mathfrak{g}[V]$ be defined by $d_e |_{\Pi \hat{S}_{\geq 1} V^\ast}=0$ and by
\[
d_e (\xi) = \frac{1}{2}\Pi \nabla (\xi) \in \Pi \hat{S}_{\geq 1}V^\ast
\]
for $\xi \in \Der_{\geq 2} (\hat{S}V^\ast )$. Then the graded Lie algebra $\mathfrak{g}[V]$ is a (pronilpotent) dgla with differential $d+d_e$, where $d$ is the induced differential from $V$.
\end{proposition}
\begin{proof}
The equations $d^2=0$ and $d_e^2=0$ are clear. To see that $d\circ d_e + d_e \circ d=0$, first note that it clearly holds for $f\in \Pi \hat{S}_{\geq 1} V^\ast$. It remains to show the equation holds for $\xi \in \Der_{\geq 2} (\hat{S} V^\ast)$. Note that $d$ is traceless and so $\nabla (d)=0$. Thus,
\begin{align*}
2(d\circ d_e + d_e \circ d) (\xi) &= d (\Pi \nabla (\xi)) + \Pi \nabla ([d,\xi]) \\
&= - \Pi d(\nabla(\xi)) + \Pi d(\nabla (\xi)) - (-1)^{|\xi|} \Pi \xi (\nabla (d)) \\
&=0.
\end{align*}
Here we have used the formula $\nabla ([\xi,\nu])=\xi\nabla (\nu) - (-1)^{|\xi||\nu|}\nu \nabla (\xi)$. The rest of the proof is straightforward, yet long-winded, calculation.
\end{proof}

\begin{definition}\label{def_unimodular}
Let $V$ be a super vector space. A unimodular \Linf-algebra structure on $V$ is an element belonging to the set $\MC(\mathfrak{g}[\Pi V])$.
\end{definition}

Spelling out Definition~\ref{def_unimodular}, a unimodular \Linf-algebra structure on $V$ is a pair $(m,f)$, where $m\in \Der_{\geq 2} ( \hat{S}\Pi V^\ast )$ and $f\in \hat{S}_{\geq 1} \Pi V^\ast$, such that
\[
d(m)+\frac{1}{2}[m,m]=0 \quad\text{ and }\quad d(f) + \frac{1}{2}\nabla (m) + m (f) = 0.
\]

\begin{remark}\label{rem_preserving_volume}
Given that $V$ is finite-dimensional, we have the standard volume form  $\mu$ on $\Pi V^\ast$. The conditions in the definition of a unimodular \Linf-algebra structure $(m,f)$ on $V$ are equivalent to the condition that the derivation $m$ preserves the volume form $e^f \mu$.
\end{remark}

\begin{remark}
It is clear that one can recover the structure of an \Linf-algebra from a unimodular \Linf-algebra by simply forgetting the function $f$.
\end{remark}

\begin{remark}
Unimodular \Linf-algebras may equivalently be defined as algebras over the cobar construction of the wheeled closure of the operad $\mathcal{C}om$, see \cite{granaker,merkulov_wheeled_props}.
\end{remark}

There is no notion of a unimodular \Linf-algebra ((quasi)-iso)morphism, and so, unlike \Linf-algebra structures, we must consider equivalence of unimodular \Linf-algebra structures only when the structures are supported on the same super vector space.

\begin{definition}\label{def_equiv_unimodular_L-inf}
Two unimodular \Linf-algebra structures on $V$ are said to be equivalent if they are equivalent as MC elements in the dgla $\mathfrak{g}[\Pi V]$.
\end{definition}

Recall that $\mathfrak{g}[\Pi V]$ is a pronilpotent dgla and so by a Theorem~\ref{thm_schlessinger_stasheff} the two notions of gauge equivalence and homotopy for MC elements in $\mathfrak{g}[\Pi V]$ are equivalent. Therefore, we can make Definition~\ref{def_equiv_unimodular_L-inf} explicit in two ways:
\begin{enumerate}
\item a structure $(m,f)$ is gauge equivalent to a structure $(m^\prime,f^\prime)$ if there exists an even vector field $\xi \in \Der_{\geq 2} (\hat{S}\Pi V^\ast)$ and an odd function $g\in \hat{S}_{\geq 1} \Pi V^\ast$ such that
\[
m^\prime = e^{\xi} (d+m) e^{-\xi} - d \quad\text{ and }\quad f^\prime = e^{\xi}(f) + (d+m) (g).
\]
\item a structure $(m,f)$ is homotopic to a structure $(m^\prime,f^\prime)$ if there exists $h(t)\in \MC (\mathfrak{g}[\Pi V] [t,dt])$ such that $h(0)=(m,f)$ and $h(1)=(m^\prime,f^\prime)$.
\end{enumerate}
It is the latter notion, homotopy, that will be of most use of in this paper.

\begin{proposition}\label{prop_unimodular_structure_bijection}
Let $V$ be a super vector space with a split quasi-isomorphism $\homology (V) \to V$. The set of equivalence classes of unimodular \Linf-algebra structures on $\homology (V)$ and $V$ are in bijective correspondence.
\end{proposition}
\begin{proof}
Define filtrations of $\mathfrak{g}[\Pi V]$ and $\mathfrak{g}[\Pi \homology V]$ by setting
\[
F_p \mathfrak{g}[\Pi V] = \Der_{\geq p} (\hat{S}\Pi V^\ast ) \ltimes \Pi \hat{S}_p V^\ast
\]
for $p\geq 1$ and similarly for $\mathfrak{g}[\Pi \homology (V)]$. With these filtrations the split quasi-isomorphism $\homology (V) \to V$ induces a filtered quasi-isomorphism $\mathfrak{g}[\Pi V] \to \mathfrak{g}[\Pi \homology (V)]$. It is a well-known fact that filtered quasi-isomorphisms of dglas induce isomorphisms on MC moduli sets, see \cite{getzler} or using the Koszul duality of \cite{laz_markl,maunder_unbased_rat_homo}.
\end{proof}

\begin{remark}
It is possible for dglas to be quasi-isomorphic and have differing MC moduli sets. So it is necessary to use the finer notion of a filtered quasi-isomorphism in the proof of Proposition~\ref{prop_unimodular_structure_bijection}. For example, the dgla
\[
\left\lbrace x, dx : |x|=-1, dx=-\frac{1}{2} [x,x] \right\rbrace
\]
is clearly quasi-isomorphic to the zero dgla, but the two dglas have different MC moduli sets.
\end{remark}

A unimodular \Linf-algebra is called strictly unimodular if it is of the form $(m,0)$. Braun and Lazarev \cite{braun_laz_unimodular} showed that an \Linf-algebra $(V,m)$ admits a unimodular lift if, and only if, it is equivalent to the strictly unimodular \Linf-algebra structure.

\subsection{Quantum L-infinity algebra structures}

Similar to \Linf-algebra and unimodular \Linf-algebra structures, a quantum \Linf-algebra structure is given by a MC element in a certain dgla that we now work towards defining. Fix some odd symplectic vector space $V$ with basis $\lbrace x_i,\xi_j \rbrace_{i,j\in \lbrace 1,2,\dots, n\rbrace}$ for $V^*$, where $|x_i|=0$ and $|\xi_i|=1$.

\begin{definition}\label{def_BV_laplacian}
The (odd) Laplacian acts on formal functions $f\in \hat{S} V^\ast$ by
\[
\Delta(f)=\sum^{n}_{i=1} \partial_{x_i} \partial_{\xi_i} f.
\]
\end{definition}

Note that the definition of the Laplacian does not depend on the choice of basis.

\begin{remark}
The Laplacian defines a dg BV-algebra structure on $\hat{S} V^\ast$ with differential $d$ and BV-operator $\Delta$, that is $\Delta^2=0=\Delta (1)$ and $d\Delta + \Delta d =0$. In particular, $\hat{S} V^\ast$ has the structure of a dg odd Poisson algebra with differential $d+\Delta$ and odd bracket given by
\[
[x,y]=(-1)^{|x|}\Delta (xy)-(-1)^{|x|}\Delta (x)y - x\Delta(y).
\]
For more details regarding BV-algebras and odd Poisson (or Gerstenhaber) algebras see \cite{roger_gerstenhaber}. For more details on dg BV-algebras and their generalisation to \BVinf-algebras see \cite{braun_laz_homotopy_BV}, for example.
\end{remark}


\begin{definition}\label{def_h[V]}
For a monomial $f\in \hat{S} V^*$ of degree $n$, the element $f\hbar^g$ of $\hat{S} V^* [|\hbar |]$ is said to have weight $2g + n$. Let $\mathfrak{h}[V]$ be the subspace of $\hat{S} V^* [|\hbar |]$ containing those elements of weight $> 2$.
\end{definition}

Clearly the odd dgla structure of $\hat{S} V^\ast$ descends to one on $\mathfrak{h} [V]$, where the differential is given by $d+\hbar \Delta$. The dgla $\mathfrak{h}[V]$ is important for several reasons:
\begin{itemize}
\item $\mathfrak{h}[V]$ defines a subspace of $\hat{S} V^* [|\hbar |]$ where we can make sense of the integrals we will use later in the paper;
\item $\mathfrak{h}[V]$ is pronilpotent;
\item $\mathfrak{h}[V]$ is where quantum \Linf-algebra structures on $\Pi V$ are defined.
\end{itemize}

\begin{definition}\label{def_qLinf}
Let $V$ be a super vector space equipped with an odd non-degenerate symmetric bilinear form. A quantum \Linf-algebra structure on $V$ is an even element
\[
m = m_0 + \hbar m_1 + \hbar^2 m_2  + \dots \in \mathfrak{h}[\Pi V]
\]
that satisfies the MC equation
\[
(d+\hbar\Delta) m +\frac{1}{2} [m, m]=0.
\]
\end{definition}

\begin{remark}
Equivalently, quantum \Linf-algebras can be defined as algebras over the Feynman transform of the modular closure of the cyclic operad governing commutative algebras. The weight grading restriction is a reflection of the stability condition for modular operads, cf.~\cite{getzler_kapranov_modular_operads}.
\end{remark}

\begin{remark}\label{rem_recovering_defs}
Quantum \Linf-algebra structures can be thought of as `higher genus' versions of cyclic and unimodular \Linf-algebra structures in the following sense: suppose that $m_0 + \hbar m_1 + \dots$ defines a quantum \Linf-algebra structure on $V$. The canonical derivation associated to $m_0$ defines an odd cyclic \Linf-algebra on $V$ (see \cite{loday}) --- moreover forgetting the cyclic structure defines an \Linf-algebra. Further, the derivation associated with $m_0$ paired with the function $m_1$ defines a unimodular \Linf-algebra structure on $V$.
\end{remark}

\begin{remark}
The MC equation in $\mathfrak{h}[V]$ is equivalent to the quantum master equation (QME)
\[
(d+\hbar\Delta) e^\frac{m}{\hbar}=0 \Leftrightarrow (d+\hbar\Delta) m +\frac{1}{2} [m, m]=0.
\]
It is this correspondence that ensures the interplay between \Linf-algebra structures and the path integrals of the BV-formalism.
\end{remark}

Just as with unimodular \Linf-algebras, without the correct notion of quantum \Linf-algebra morphism, the equivalence of quantum \Linf-algebra structures is restricted to those structures supported on the same super vector space.

\begin{definition}
Two quantum \Linf-algebra structures supported on the same super vector space are said to be equivalent if they are equivalent as MC elements.
\end{definition}

As with the other \Linf-algebra structures, quantum \Linf-algebra structures are MC elements in a pronilpotent dgla and so two equivalent quantum \Linf-algebra structures are both gauge equivalent and homotopic (see Theorem~\ref{thm_schlessinger_stasheff}).

\section{The doubling constructions}\label{sec_doubling}

The doubling constructions (first given in \cite{braun_laz_unimodular}) embed the dglas where \Linf-algebra structures live into `richer' dglas, allowing us to make use of the BV-formalism. In this section we recall these constructions.

\subsection{Doubles of vector spaces}

Let $V$ be some super vector space. Consider the super vector space $V^\ast \oplus V$. This super vector space can be endowed with a non-degenerate anti-symmetric bilinear form $\langle,\rangle$ given by
\begin{itemize}
\item letting $V$ and $V^\ast$ be isotropic subspaces;
\item setting $\langle v^\ast , u \rangle=v^\ast (u)$, where $v^\ast \in V^\ast$ and $u\in V$;
\item extending in the obvious manner.
\end{itemize}
Similarly, the super vector space $V^\ast \oplus \Pi V$ can be endowed with a non-degenerate symmetric odd bilinear form $(,)$ by defining
\begin{itemize}
\item letting $\Pi V$ and $V^\ast$ be isotropic subspaces;
\item setting $\langle v^\ast , \Pi u \rangle=(-1)^{|v^\ast|} v^\ast (u)$, where $v^\ast \in V^\ast$ and $u\in V$;
\item extending in the obvious manner.
\end{itemize}

\begin{definition}
Let $V$ be a super vector space. The direct sum $V^\ast \oplus V$ is called the even double of $V$ and the direct sum $V^\ast \oplus \Pi V$ is called the odd double of $V$.
\end{definition}

\begin{proposition}
The spaces $V^\ast \oplus V$ and $V^\ast \oplus \Pi V$ are linear symplectic spaces (even and odd, respectively). Therefore, $\hat{S}(V^\ast \oplus V)$  and $\hat{S}(V^\ast \oplus \Pi V)$ have associated Poisson brackets, making the former a dgla and the latter an odd dgla. \qed
\end{proposition}

\begin{remark}
The bracket on $\hat{S} (V^\ast \oplus \Pi V)$ is sometimes called the antibracket.
\end{remark}

The constant functions define one-dimensional centres of the Lie algebras $\hat{S} (V^\ast \oplus V)$ and $\hat{S} (V^\ast \oplus \Pi V)$. The corresponding quotient algebras $\hat{S}_+ (V^\ast \oplus V)$ and $\hat{S}_+ (V^\ast \oplus \Pi V)$ can be identified with the Lie algebras of symplectic vector fields on $V^\ast \oplus V$ and $V^\ast \oplus \Pi V$, respectively. The Lie brackets on $\hat{S}_+ (V^\ast \oplus V)$ and $\hat{S}_+ (V^\ast \oplus \Pi V)$ are the commutators of the corresponding Hamiltonians and will be denoted by $(,)$ and $\lbrace,\rbrace$, respectively (see \cite{hamilton_laz_noncomm_geo} for more details).

The Lie bracket $(,)$ on $\hat{S}_+ (V^\ast \oplus V)$ is a derivation of the product: for $f,g,h \in \hat{S}_+ (V^\ast \oplus V)$ we have
\[
(f,gh)=(f,g)h + (-1)^{|f||g|} g (f,h).
\]
Similarly, the bracket $\lbrace,\rbrace$ on $\hat{S}_+ (V^\ast \oplus \Pi V)$ is an odd Lie bracket in that:
\begin{itemize}
\item $\lbrace,\rbrace$ is symmetric: for $f,g \in \hat{S}_+ (V^\ast \oplus \Pi V)$ we have
\[
\lbrace f,g \rbrace = (-1)^{|f||g|} \lbrace g,f \rbrace;
\]
\item $\lbrace,\rbrace$ satisfies the odd Jacobi identity: for $f,g,h \in \hat{S}_+ (V^\ast \oplus \Pi V)$ we have
\[
\lbrace f,\lbrace g,h \rbrace \rbrace = (-1)^{|f|+1} \lbrace \lbrace f,g \rbrace, h \rbrace + (-1)^{(|f|+1)(|g|+1)} \lbrace g,\lbrace f,h \rbrace \rbrace.
\]
\end{itemize}
Moreover, the bracket $\lbrace,\rbrace$ is an odd derivation of the product: for $f,g,h \in \hat{S}_+ (V^\ast \oplus \Pi V)$ we have
\[
\lbrace f,gh \rbrace =\lbrace f,g \rbrace h + (-1)^{(|f|+1)|g|} g \lbrace f,h \rbrace.
\]
Explicitly, given a basis $x_1,\cdots, x_n$ in $V$, the Poisson brackets for linear functions are given by
\[
(x_i^\ast,x_j)=\delta_{ij}=\lbrace x_i^\ast ,x_j \rbrace.
\]

\subsection{The doubling constructions}

The doubling constructions associate to any formal vector field on $V$ a Hamiltonian on each of the even and odd doubles of $V$. Before recalling their definitions, note that any derivation of $\hat{S}V^\ast$ is uniquely determined by its value on $V^\ast$. Therefore, we have an isomorphism
\[
\Der \left(\hat{S} V^\ast \right) \cong \hat{S} V^\ast \otimes V
\]
of super vector spaces. The obvious embedding $V\hookrightarrow \hat{S}V$ therefore induces a map
\[
\Dev \colon \Der (\hat{S}V^\ast)\cong \hat{S} V^\ast \otimes V \hookrightarrow \hat{S}V^\ast \otimes \hat{S}_+ V \subset \hat{S}_+ (V^\ast \oplus V).
\]
Similarly, a map
\[
\Dod \colon \Der (\hat{S}V^\ast)\cong \hat{S} V^\ast \otimes V \hookrightarrow \hat{S}V^\ast \otimes \hat{S}_+ \Pi V \subset \hat{S}_+ (V^\ast \oplus \Pi V)
\]
is given by the formula
\[
\hat{S}V^\ast \otimes V \ni f\otimes v \mapsto (-1)^{|f|} f \otimes \Pi v \in \hat{S}V^\ast \otimes \hat{S}_+ \Pi V \subset \hat{S}_+ (V^\ast \oplus \Pi V).
\]

\begin{definition}
$\Dev$ and $\Dod$ are called the even and odd doubling maps, respectively.
\end{definition}

The proof of the following Theorem can be found in the original paper \cite{braun_laz_unimodular}.

\begin{theorem}[Braun and Lazarev]
The map $\Dev$ is a morphism of dglas, that realises $\Der (\hat{S}V^\ast )$ as a sub-dgla in the dgla of formal Hamiltonians on $V^\ast \oplus V$. Moreover, the map $\Dod$ is an odd morphism of dglas, that realises $\Der (\hat{S}V^\ast )$ as a sub-dgla in the odd dgla of formal Hamiltonians on $V^\ast \oplus \Pi V$
\end{theorem}

A derivation $X_f$ of $\hat{S} \left(V^\ast \oplus V \right)$ is associated to every function $f\in \hat{S} \left(V^\ast \oplus V \right)$. The derivation $X_f$ is known as the Hamiltonian vector field associated to $f$. Explicitly, the action of $X_f$ on a function $g\in \hat{S} (V^\ast \oplus V)$ is given by $X_f (g) = (-1)^{|f|} (f,g)$. Clearly the constant functions act by zero. This association is well-known to give a Lie algebra morphism
\[
\phi_{\text{ev}} \colon \hat{S}_+ \left(V^\ast \oplus V \right) \to \Der \left( \hat{S} \left( V^\ast \oplus V \right) \right).
\]
Similarly, there is an odd Lie algebra morphism
\[
\phi_{\text{od}} \colon \hat{S}_+ \left(V^\ast \oplus \Pi V \right) \to \Der \left( \hat{S} \left( V^\ast \oplus \Pi V \right) \right).
\]
Pre-composing the above morphisms $\phi_{\text{ev}}$ and $\phi_{\text{od}}$ with the doubling constructions we obtain the following two morphisms of Lie algebras:
\[
\Der \left(\hat{S}V^\ast\right) \to \Der \left( \hat{S} \left( V^\ast \oplus V \right)\right) \quad\text{ and }\quad \Der \left(\hat{S}V^\ast\right) \to \Der \left( \hat{S} \left( V^\ast \oplus \Pi V \right)\right).
\]
By an abuse of notation, we will often refer to the maps that realise $\Der (\hat{S} V^\ast )$ as a Lie subalgebra in both $\Der ( \hat{S} ( V^\ast \oplus V ))$ and $\Der  \hat{S} ( V^\ast \oplus \Pi V ))$ as the even and odd doubling maps. Moreover, we will continue to use the notation $\Dev$ and $\Dod$ for these maps and the already established doubling maps. Context should make it apparent which doubling map is used.

Since \Linf-algebra structures are given by MC elements and the doubling maps are dgla morphisms, we make the following definition.

\begin{definition}
Let $m\in \Der_{\geq 2} (\hat{S} \Pi V^\ast )$ be an \Linf-algebra structure on the super vector space $V$.
\begin{itemize}
\item The even double of $(V,m)$ is the \Linf-algebra structure
\[
\Dev (m) \in \Der_{\geq 2} \left( \hat{S} \Pi \left(V^\ast \oplus V \right)\right).
\]
\item The odd double of $(V,m)$ is the \Linf-algebra structure
\[
\Dod (m) \in \Der_{\geq 2} \left( \hat{S} \Pi \left(V^\ast \oplus \Pi V \right)\right).
\]
\end{itemize}
\end{definition}

The even and odd doubles of an \Linf-algebra structure are naturally cyclic \Linf-algebras (even and odd, respectively). In fact, the even double of an \Linf-algebra structure is actually a strictly unimodular \Linf-algebra structure. Before we prove this, we need the following preliminary result regarding the divergence of the even and odd doubles.

\begin{proposition}\label{prop_div_and_doubles}
For $\xi \in \Der \left( \hat{S} V^\ast \right)$, $\nabla (X_{\Dev (\xi)} )=0$ and $\nabla (X_{\Dod (\xi)} )=2\nabla (\xi)$.
\end{proposition}
\begin{proof}
Choose some basis $x_1,\cdots,x_n$ in $V^\ast$ with dual basis $x_1^\ast,\cdots x_n^\ast$ in $V$. If $\xi=f\partial_{x_i}$ for some $i$, then
\begin{align*}
\nabla (X_{\Dev (\xi)} ) &= \nabla (X_{fx_i^\ast}) \\
	&= \nabla \left((fx_i^\ast,-)\right) \\
	&= \nabla \left(\sum_{j} \left(\partial_{x_j^\ast}(fx_i^\ast) \partial_{x_j} - (-1)^{|x_j||x_j^\ast|} \partial_{x_j}(fx_i^\ast) \partial_{x_j^\ast} \right)\right) \\
	&= \nabla \left(\sum_{j} \left((-1)^{|f||x_j^\ast|} f\partial_{x_j^\ast}(x_i^\ast) \partial_{x_j} - (-1)^{|x_j||x_j^\ast|} \partial_{x_j}(f)x_i^\ast \partial_{x_j^\ast} \right)\right) \\
	&= \partial_{x_i} (f) - \partial_{x_i}(f) = 0.
\end{align*}
The odd case is almost analogous to the even one.
\end{proof}

Therefore, the next result immediately follows.

\begin{proposition}\label{prop_even_double_is_unimodular}
Let $(V,m)$ be an \Linf-algebra. The even double $(V\oplus V^\ast,X_{\Dev(m)})$ defines the strictly unimodular \Linf-algebra structure $(X_{\Dev(m)},0)$ on $V\oplus V^\ast$.
\end{proposition}
\begin{proof}
The statement follows from Proposition~\ref{prop_div_and_doubles}, because the conditions for $(X_{\Dev(m)},0)$ to be a unimodular \Linf-algebra structure are exactly the MC equation for $\Dev(m)$ and the equation $\nabla (X_{\Dev (\xi)} )=0$, which are both clearly satisfied.
\end{proof}

Naturally, it is also possible to double unimodular \Linf-algebra structures.

\begin{definition}
Let $(m,f)$ be a unimodular \Linf-algebra structure on the super vector space $V$.
\begin{itemize}
\item The even double of $(m,f)$ is the unimodular \Linf-algebra structure
\[
(\Dev (m),f) \in \mathfrak{g} [\Pi \left(V^\ast \oplus V \right)].
\]
\item The odd double of $(m,f)$ is the unimodular \Linf-algebra structure
\[
(\Dod (m),f) \in \mathfrak{g} [\Pi \left(V^\ast \oplus \Pi V \right)].
\]
\end{itemize}
Here $f$ is defined via the pullback of the obvious inclusions.
\end{definition}

In fact, the odd double of a unimodular \Linf-algebra structure readily permits extension to a quantum \Linf-algebra structure. Compare this with how the even double of an \Linf-algebra structure readily extended to a (strictly) unimodular \Linf-algebra structure (Proposition~\ref{prop_even_double_is_unimodular}).

\begin{proposition}\label{prop_unimodular_to_quantum}
Let $(m,f)$ be a unimodular \Linf-algebra structure on a super vector space $V$. The formal function $\Dodq(m,f):=\Dod(m)+\hbar f \in \mathfrak{h}[\Pi (V^\ast \oplus \Pi V)]$ defines a quantum \Linf-algebra structure on $V^\ast \oplus \Pi V$, where $f$ is extended via the pullback of the inclusion $V\to V\oplus \Pi V^\ast$.
\end{proposition}
\begin{proof}
The MC equation for $\Dod(m)+\hbar f$ splits into three equations: the constant, linear, and quadratic parts in $\hbar$. The first two equations are precisely the ones that define the unimodular \Linf-algebra structure and are hence satisfied. The third is $\Delta (f)=0$, which clearly holds as $f$ is defined via the pullback.
\end{proof}

Thus, we have the following consequence of Proposition~\ref{prop_even_double_is_unimodular} and Proposition~\ref{prop_unimodular_to_quantum}.

\begin{lemma}\label{lem_quadruple}
Let $m$ be a \Linf-algebra structure on a super vector space $V$. The formal function $\Dodq\circ\Dev (m)$ defines a quantum \Linf-algebra structure on $(V\oplus V^\ast) \oplus \Pi (V\oplus \Pi V^\ast)$. \qed
\end{lemma}

\begin{remark}
At first, `quadrupling' like this may seem unnatural, but this procedure provides the richer structure necessary for integration --- exactly like the addition of (anti-)ghosts and anti-fields in quantum field theory, see Section~\ref{sec_ghosts_antifields} for a more direct comparison.
\end{remark}

\begin{proposition}\label{prop_doubles_respect_homotopy}
Let $m_1$ and $m_2$ be two equivalent \Linf-algebra structures on a super vector space $V$. Then the following are equivalent \Linf-algebra structures:
	\begin{itemize}
	\item $\Dev (m_1)$ and $\Dev (m_2)$;
	\item $\Dod (m_1)$ and $\Dod (m_2)$.
	\end{itemize}
Similarly, if $(m_1,f_1)$ and $(m_2,f_2)$ be two equivalent unimodular \Linf-algebra structures on $V$, then the following are equivalent unimodular \Linf-algebra structures:
	\begin{itemize}
	\item $(\Dev (m_1),f_1)$ and $(\Dev (m_2),f_2)$;
	\item $(\Dod (m_1),f_1)$ and $(\Dod (m_2),f_2)$.
	\end{itemize}
Moreover, the quantum \Linf-algebra structures $\Dodq (m_1,f_1)$ and $\Dodq (m_2,f_2)$ are equivalent.
\end{proposition}
\begin{proof}
By extending the relevant doubling map $t,dt$-linearly, the homotopy in the domain is mapped to one in the image.
\end{proof}

\subsection{Quadruples and field theory}\label{sec_ghosts_antifields}

Here we will explain how the `quadruple' vector space of Lemma~\ref{lem_quadruple} can be interpreted using (anti-)ghosts and anti-fields of quantum field theory. Our first step is to recall the BRST construction \cite{BRS_renormalisation_of_gauge, tyutin}. In particular, we will recall the homological treatment of this construction after explaining its motivation as a symplectic reduction. The aim of the BRST construction is to replace a gauge symmetry by a rigid symmetry, which is present after fixing the gauge. To achieve this, extra fields are added to the theory and the BRST symmetry is a nilpotent differential on this enlarged space. Hence, the BRST symmetry defines a cohomology theory, known as BRST cohomology. Moreover, we have that the zeroeth cohomology of this theory is given precisely by the gauge invariant functions (the observables) of the original theory.

\subsubsection{\texorpdfstring{\Linf-algebras}{Homotopy Lie algebras} as `ghost' field theories}

In order to compare the doubling constructions with the addition of (anti-)ghosts and anti-fields, we must place the data of an \Linf-algebra into the realm of a gauge theory without adding too much in the way of artificial auxiliary constructions. Let us, therefore, consider the rather simplified collection of data as a gauge theory:
	\begin{enumerate}
	\item the phase space is a point $M = \lbrace 0 \rbrace$;
	\item the algebra of functions on $M$ is therefore the ground field $C^{\infty} (M) = \mathbb{R}$;
	\item the action functional $M\to \mathbb{R}$ is the zero map;
	\item the gauge symmetry is given by an \Linf-algebra structure $m$ on a super vector space $V$, but without any Hamiltonian action $\delta \colon V\to C^{\infty} (M)$.
	\end{enumerate}
In short, the only non-trivial piece of the above data is the \Linf-algebra. (Hamiltonian formalism.)

The data above defines the momentum map (or moment map\footnote{The term moment map is a misnomer, stemming from a mistranslation of the french for ``application moment''.}) $\Phi \colon M \to V^\ast$, which is defined by
\[
\Phi (m) (v) = \delta (v) (m).
\]
Of course, in this particular case the momentum map is rather boring -- -the image contains one linear function $\Phi (0)$ which maps any element of $v$ to $\delta (v)\in \mathbb{R}$. The momentum map is the geometric generalisation of a conserved quantity, such as linear momentum or angular momentum, and it is equivariant with respect to the coadjoint action of $V$ on $V^\ast$.

\subsubsection{Symplectic reduction}

We will now briefly discuss the idea of a symplectic reduction, before indicating how this can be realised using homological algebra using the BRST construction.

In a general field theory, the algebra of functions $C^\infty (M)$ has a symplectic structure --- in is important to note that the BRST construction never makes essential use of this structure, but it can be helpful for calculations. Moreover, the pre-image of the momentum map $\Phi$ at a regular value $\alpha \in V^\ast$ defines a coisotropic submanifold $C$ of $M$. The theory should be considered on $C$ rather than $M$, as physical quantities are conserved on $C$. For trivial case at hand, the only non-empty pre-image is $M = \lbrace 0 \rbrace$.

The isotropy group of $\alpha$ under the coadjoint action is given by
\[
G_{\alpha} = \lbrace g \in G : g(\alpha) = \alpha \rbrace.
\]
The equivariance of $\Phi$ ensures that the quotient $B:=C/G_\alpha$ is well-defined. The quotient $B$ is known as the reduced phase space (corresponding to the momentum value $\alpha$). The fact that this space ends up with a symplectic structure induced from the original one on $M$ is known as the Marsden-Weinstein-Meyer Theorem \cite{marsden_weinstein, meyer}, which is a special case of the more general symplectic reduction. In the particular setting of the above simple data, both the coisotropic submanifold and the quotient are the point $\lbrace 0 \rbrace$.

Let us suppose for now that the \Linf-algebra is in fact a Lie algebra with associate Lie group $G$. Looking at these steps a little closer, we have
\[
C^{\infty} (C) = C^{\infty} (M) / \langle \delta (V) \rangle,
\]
where $\langle \delta (V) \rangle$ is the ideal generated by the image of $\delta$ (which is exactly the ideal of functions that vanish on $C$). Next, a function on the quotient $C/G_\alpha$ can be thought of as a function on $C$ that is invariant under the action of $G$. Hence, we have
\[
C^{\infty} (B) = C^{\infty} (C)^{G},
\]
where the superscript $G$ indicates the $G$-invariants. Therefore, we see there are two steps in this symplectic reduction process:
	\begin{enumerate}
	\item restriction to $C$, where quantities are conserved;
	\item projecting onto the quotient $B$, dealing with orbits of the symmetry.
	\end{enumerate}
	
\subsubsection{The BRST construction}
	
We will now cast these two steps into the BRST construction \cite{barnich_brandt_henneaux, figueroa_kimura, kostant_sternberg}. To begin with in the BRST construction, one would ordinarily construct a differential $\delta$ from the action the \Linf-algebra structure by extending the linear map $V\to C^{\infty} (M)$ to an odd derivation $\delta \colon \hat{S}^{q} \Pi V \otimes C^{\infty} (M) \to \hat{S}^{q-1} \Pi V \otimes C^{\infty} (M)$. It is clear that this forces $\delta^2=0$ and so this extension, therefore, defines a chain complex
\[
\cdots \xrightarrow{\delta} \hat{S}^{2} \Pi V \otimes C^{\infty} (M) \xrightarrow{\delta} \Pi V \otimes C^{\infty} (M) \xrightarrow{\delta} C^{\infty} (M) \to 0.
\]
This construction is essentially the one due to Koszul \cite{koszul} (and a related less-restrictive construction is due to Tate \cite{tate}). The differential $\delta$ is known as the Koszul (or Koszul-Tate) differential and the complex is known as the Koszul(-Tate) resolution, because it defines a resolution of functions on $C$. That is, the homology of the chain complex is trivial in all degrees except degree $0$, where the homology is precisely the functions on $C$. In physics the $\mathbb{Z}$ grading of this chain complex is known as the ghost number. Note that we also have a $\mathbb{Z}_2$ grading, and so this complex is $\mathbb{Z} \times \mathbb{Z}_2$ graded. The elements of $V$ are known as ghosts and in the Tate construction the extra cochains are known as antighosts. For us, however, there is no need for this extension and we have $\hat{S} \Pi V \otimes C^{\infty} (M) \cong \hat{S} \Pi V$. The differential $\delta$ acts as the identity on linear components and is extended to general tensors as an odd derivation.

The next step is to consider $\hat{S} \Pi V \otimes C^{\infty} (M)$ as a module for the \Linf-algebra structure $m$ on $V$. We can then consider the Chevalley-Eilenberg complex for the Lie algebra cohomology with coefficients in the module $\hat{S} \Pi V \otimes C^{\infty} (M)$, which is
\[
\hat{S} \Pi V^\ast \otimes \hat{S} \Pi V \otimes C^{\infty} (M) \cong \hat{S} \Pi V^\ast \otimes \hat{S} \Pi V,
\]
with the usual Chevalley-Eilenberg differential $d_{CE}$. The zeroth cohomology is the space of invariants of the action on the module $\hat{S} \Pi V \otimes C^{\infty} (M)$, or in other words the functions on $B$.

Extending the differential $\delta$ to $\delta:=\id \otimes \delta$ we have that $d_{CE}^2=0$, $\delta^2=0$, and $\delta \circ d_{CE} = d_{CE} \circ \delta$. Hence, we have the data of a double complex
\begin{center}
\begin{tikzpicture}
\matrix(m)[matrix of math nodes, column sep = {0.8em,between borders}, row sep = 2em]{
 & \vdots &[1.4em] \vdots \\
\cdots & \hat{S}^{p} \Pi V^\ast \otimes \hat{S}^{q} \Pi V \otimes C^{\infty} (M) & \hat{S}^{p} \Pi V^\ast \otimes \hat{S}^{q-1} \Pi V \otimes C^{\infty} (M) & \cdots \\
\cdots & \hat{S}^{p+1} \Pi V^\ast \otimes \hat{S}^{q} \Pi V \otimes C^{\infty} (M) & \hat{S}^{p+1} \Pi V^\ast \otimes \hat{S}^{q-1} \Pi V \otimes C^{\infty} (M) & \cdots \\
 & \vdots & \vdots \\
};
\path[->]
(m-1-2) edge (m-2-2)
(m-1-3) edge (m-2-3)
(m-2-1) edge (m-2-2)
(m-2-3) edge (m-2-4)
(m-3-1) edge (m-3-2)
(m-3-3) edge (m-3-4)
(m-3-2) edge (m-4-2)
(m-3-3) edge (m-4-3)
(m-2-2) edge node[above] {$\delta$} (m-2-3)
(m-2-2) edge node[left] {$d_{CE}$} (m-3-2)
(m-3-2) edge node[above] {$\delta$} (m-3-3)
(m-2-3) edge node[left] {$d_{CE}$} (m-3-3)
;
\end{tikzpicture}
\end{center}
Accordingly, we can construct the total complex 
\[
\operatorname{Tot}_{k} = \bigoplus_{p-q=k} \hat{S}^{p} \Pi V^\ast \otimes \hat{S}^{q} \Pi V \otimes C^{\infty} (M) \cong \bigoplus_{p-q=k} \hat{S}^{p} \Pi V^\ast \otimes \hat{S}^{q} \Pi V
\]
with total differential $D=d_{CE} + \delta \colon \operatorname{Tot}_{k} \to \operatorname{Tot}_{k+1}$.

We also have that the zeroth cohomology of this total complex is given by $C^{\infty} (B)$. In this sense, the BRST has provided a homological way to view the functions on the reduced phase space.

The total differential $D$, which is the ``classical'' BRST operator (or BRST symmetry), is precisely $\Dev (m)$ and as such is an inner derivation. This observation is well-known even when $M$ is more interesting. In this sense, the symmetry has been encoded as a Hamiltonian function on the space of fields, ghosts, and anti-ghosts. In fact, the classical BRST operator $\Dev (m)$ is precisely the one found by Batalin-Vilkovisky \cite{batalin_vilkovisky_S-matrix}. We note that in the general case, the construction of $D$ may not square to zero, but this can be fixed using homological perturbation theory \cite{henneaux}.

The poisson bracket is degree one in the ghost number.

Here we have the ghost fields $V$ and $V^\ast$. In physics, ghosts are meant to describe non-physical fields, and so it is fitting that our field theory is made entirely of ghosts, meaning that it has no physical meaning.


All in all, we are considering the algebra of functions $\hat{S}\Pi (V\oplus V^\ast)^\ast$ on the manifold $V\oplus V\ast$ of ghost fields.



One idea in QFT is to quantize the algebra $\hat{S} \Pi (V\oplus V^\ast)$ by constructing the Clifford algebra $C(V \oplus V^\ast)$ using the canonical quadratic form given by evaluation of $V^\ast$ on $V$. However, this does not fit with the doubling constructions. Another idea, and one that fits with the doubling constructions, is the BV-formalism.

\subsubsection{An alternative field theory}

Alternatively, we realise an \Linf-algebra structure $m$ on a super vector space $V$ as a field theory by setting:
	\begin{enumerate}
	\item $\Pi V^\ast$ as the space of fields;
	\item the zero function $\hat{S} \Pi V^\ast \to \mathbb{R}$ as the action functional;
	\item the coadjoint action as the set of gauge symmetries.
	\end{enumerate}
	
Therefore, similar to before, we construct a resolution of the action-invariant functions on fields, $\hat{S} \Pi V$, i.e.~we construct the Chevalley-Eilenberg cochain complex
\[
\hat{S} \Pi V^\ast \otimes \hat{S} \Pi V \cong \hat{S} \Pi (V\oplus V^\ast)^\ast.
\]
As before, the differential on this complex is precisely that of $\Dev(m)$.

\subsubsection{Quantising and the BV-formalism}

So far we have recast a classical theory and a reduction into homological algebra. Quantising classical theories in the BRST formalism involves Clifford algebras and can be a fruitful technique. However, we will quantise the theory using the BV-formalism instead as it aligns with the doubling construction and hence our interpretation of doubling via QFT.

So far, interpreting the classical ghost theory using the path integral formalism, we wish to make sense of integrals over $V\oplus V^\ast$ against the measure $\mu$. In one sense we are in a good situation, the derivation $\Dev(m)$ preserves the measure $\mu$. There is, however, a problem: there is no action functional to apply perturbative techniques of integral evaluation to and even if there was this functional would be highly degenerate on $V\oplus V^\ast$. The BV-formalism allows us to rectify this issue.

According to the BV-formalism, we take the shifted cotangent bundle of the space of fields. The functions on the shifted cotangent space is the ring of polyvector fields. Taking the shifted cotangent bundle results in an odd symplectic space with canonical BV-laplacian that is compatible with the measure. Explicitly for the `ghost theory' constructed in the previous section, the shifted cotangent bundle of the (ghost) fields $V\oplus V^\ast$ is
\[
\Pi \operatorname{T}^\ast (V\oplus V^\ast) \cong (V\oplus V^\ast) \oplus \Pi (V\oplus V^\ast)^\ast.
\]
Clearly, as a vector space, this is the odd double of the even double. We call the latter two summands $\Pi (V\oplus V^\ast)^\ast$ the anti-ghosts. If there were also non-ghost fields, these would also have anti-fields. The algebra of functions on the shifted cotangent space is
\[
\hat{S} \Pi ((V\oplus V^\ast) \oplus \Pi (V\oplus V^\ast)^\ast)^\ast.
\]
Transferring the \Linf-algebra structure from $V\oplus V^\ast$ to its shifted cotangent space is done in precisely the fashion of the odd double, meaning the odd double lifts to a quantum \Linf-algebra structure $S_0 := \Dodq \circ \Dev (m)$. Explicitly, we consider the function $S_0$ as an action functional, the BV-action. When we begin with a non-trivial action functional $f$ on $V\oplus V^\ast$, the resulting BV-action functional is the sum of the original action functional $f$ (extended via pullback) and the action functional $S_0$. The action functional is $f$ preserved by the vector field $\Dev(m)$ and so it is clear that $f+S_0$ satisfies the classical master equation
\[
d(f+S_0) + \frac{1}{2}\lbrace f + S_0, f + S_0 \rbrace = 0.
\]
In this sense, we are still considering the `classical' BV complex. That is, we are considering the complex
\[
\hat{S} \Pi ((V\oplus V^\ast) \oplus \Pi (V\oplus V^\ast)^\ast)^\ast.
\]
with differential $d + \lbrace S_0 , - \rbrace$. Putting this another way, we have the Chevalley-Eilenberg complex on the cotangent bundle of the \Linf-algebra of the derivation $\lbrace S_0 , - \rbrace$. In quantising this algebraic structure, we introduce the formal parameter $\hbar$ and the BV-Laplacian $\hbar \Delta$, so that the complex we are considering is 
\[
\hat{S} \Pi ((V\oplus V^\ast) \oplus \Pi (V\oplus V^\ast)^\ast)^\ast [|\hbar|]
\]
and the differential is $d + \hbar \Delta + \lbrace S_0 , - \rbrace$. This differential is sometimes known as the quantum BV operator and this quantisation is known as cotangent quantisation \cite{gwilliam_thesis}. Now, since $f\equiv 0$ and the function $S_0$ is harmonic (coming from a unimodular \Linf-algebra structure) we automatically have that $S_0$ satisfies the quantum master equation
\[
(d+\hbar \Delta) e^{\frac{S_0}{\hbar}} = 0 \left( \Leftrightarrow (d+\hbar \Delta) (S_0) + \frac{1}{2} \lbrace S_0 , S_0 \rbrace = 0 \right).
\]
In the general case, extending a solution to the classical master equation is no easy task and usually one attempts to introduce higher order terms step-by-step. The fact that the $d + \hbar \Delta + \lbrace S_0 , - \rbrace$ squares to zero is a consequence of the following facts: $d$ and $\Delta$ commute; the odd dgla structure of $\hat{S} \Pi ((V\oplus V^\ast) \oplus \Pi (V\oplus V^\ast)^\ast)^\ast [|\hbar|]$; the quantum master equation for $S_0$.


This step converts our artificial classical field theory (coming from the original \Linf-algebra structure $m$) into an action functional that satisfies the quantum master equation and hence a quantum field theory. In particular, the constructed action functional can be integrated using perturbation theory to produce an effective field theory (which is the minimal model of the quantum \Linf-algebra structure). In this way, the even double creates a classical theory that is then quantised by applying the odd double. Moreover, this process is such that there is no difficulty in constructing the quantum action from the classical one, as the classical action is harmonic (meaning all steps in the step-by-step construction of the quantum action in $\hbar$ become trivial).

By the well-known BV Stokes' Theorem (Proposition~\ref{prop_BV_stokes}), we can see that for any $\Delta$-closed function the expectation is independent of continuous deformations of Lagrangian subspace and if it is $\Delta$-exact, then its expectation is zero. Therefore, the $\Delta$-cohomology on
\[
\hat{S} \Pi ((V\oplus V^\ast) \oplus \Pi (V\oplus V^\ast)^\ast)^\ast.
\]
controls the gauge-fixing of path integrals.

\section{Halving}\label{sec_halving}

We will now define one-sided inverses to the doubling maps, which we call the `halving maps' --- these maps are the obvious maps. Despite the doubling maps being morphisms of dglas, the halving maps are not maps of dglas. This turns out to not be a problem, because we are only interested in Lie subalgebras of the doubles where the halving maps are in fact morphisms of dglas.

\begin{definition}
The even halving map is the map
\[
\Hev \colon \hat{S}_+ \left(V^\ast \oplus V \right) \to \hat{S}_+ V^\ast \otimes V \cong \Der \left(\hat{S}_+ V^\ast \right),
\]
given by $f\otimes v \mapsto f\otimes v$ for all $f\in \hat{S}_+ V^\ast$ and $v\in V$, and zero otherwise.

Similarly, the odd halving map
\[
\Hod \colon \hat{S}_+ \left(V^\ast \oplus \Pi V \right) \to \hat{S}_+ V^\ast \otimes V \cong \Der \left(\hat{S}_+ V^\ast \right)
\]
is given by $f\otimes \Pi v \mapsto (-1)^{|f|} f\otimes v$ for all $f\in \hat{S}_+ V^\ast$ and $v\in V$, and zero otherwise.
\end{definition}

Unlike like the doubling maps, the halving maps are only morphisms of super vector spaces. However, the halving maps `undo' the doubling and restricted the to the image of the doubling maps the halving maps are morphisms of dglas.

\begin{proposition}\label{prop_H_is_one_sided_inverse_and_when_restricted_is_dgla}
We have that $\Hev \circ \Dev = \id$ and $\Hod \circ \Dod=\id$, where $\id$ is the identity map of $\Der (\hat{S}_+ V^\ast )$. Moreover, the restrictions $\Hev\mkern-8mu \mid_{\im(\Dev)}$ and $\Hod\mkern-8mu \mid_{\im(\Dod)}$ are morphisms of dglas (even and odd, respectively).
\end{proposition}
\begin{proof}
The proof is a straightforward exercise in chasing definitions.
\end{proof}

Therefore, in general the halving maps will not map \Linf-algebra structures on $V\oplus V^\ast$ or $V\oplus \Pi V^\ast$ to \Linf-algebra structures on $V$. However, restricting the domain somewhat fixes this issue.

\begin{proposition}\label{prop_halving_unimodular}
If $m$ belongs to the image of $\Dev$ and defines an \Linf-algebra structure on $V\oplus V^\ast$, then $\Hev(m)$ defines an \Linf-algebra structure on $V$. Further, if $f\in \hat{S} \Pi V^\ast \subset \hat{S} \Pi (V\oplus V^\ast)^\ast$ is such that the pair $(m,f)$ defines a unimodular \Linf-algebra structure on $V\oplus V^\ast$, then $(\Hev (m),f)$ defines a unimodular \Linf-algebra structure on $V$. The analogous statement also holds for $\Dod$.
\end{proposition}
\begin{proof}
The proposition follows from Proposition~\ref{prop_H_is_one_sided_inverse_and_when_restricted_is_dgla}.
\end{proof}

\begin{proposition}\label{prop_havling_quantum_to_unimodular}
Let $m=m_0+\hbar m_1 + \cdots$ be a quantum \Linf-algebra structure on $V\oplus \Pi V^\ast$ such that $m_0\in \hat{S} \Pi V^\ast \otimes V$ and $m_1 \in \hat{S} \Pi V^\ast$. The pair $\Hodq (m):=(X_{\Hod(m_0)},\Hod(m_1))$ defines a unimodular \Linf-algebra structure on $V$.
\end{proposition}
\begin{proof}
The constant part (in $\hbar$) of the QME for $m$ is 
\[
d (m_0) + \frac{1}{2}[m_0,m_0]=0
\]
and the $\hbar$-linear part of the QME for $m$ is
\[
d(m_1) + \Delta (m_0) + [m_1,m_0] =0.
\]
So, the result follows from Proposition~\ref{prop_halving_unimodular}.
\end{proof}

Just as with halving structures, it is also necessary to ask that homotopies of such structures also belong to the correct subspace before halving. In particular, a homotopy could generally belong to a larger subspace and a simple truncation will no longer automatically satisfy the MC equation.

\begin{proposition}\label{prop_halving_respects_homotopy}
Let $m_1$ and $m_2$ be homotopic \Linf-algebra structures on $V^\ast \oplus V$ via homotopy $h(t)$ in the image of $\Dev$. Then $\Hev (m_1)$ and $\Hev (m_2)$ are homotopic via $\Hev (h(t))$. Similarly for \Linf-algebra structures on $V^\ast \oplus \Pi V$ and $\Hod$.
\end{proposition}
\begin{proof}
Extending the halving map $t$,$dt$-linearly, we see that $\Hev (h(t))$ defines an \Linf-algebra structure on $V$. Moreover, specialising to $t=0$ and $t=1$ recovers $\Hev (m_1)$ and $\Hev (m_2)$, respectively. Hence, $\Hev (h(t))$ defines a homotopy of the \Linf-algebra structures $\Hev (m_1)$ and $\Hev (m_2)$. The odd result is analogous.
\end{proof}

\section{The BV-formalism}\label{sec_BV}

\subsection{Strong deformation retracts}\label{sec_decomp}

When a vector space is equipped with a strong deformation retract (SDR) onto some choice of representatives for homology, one has a canonical choice of isotropic/Lagrangian subspace. This isotropic subspace is crucial to the integration of the BV-formalism. The notion of a SDR from a space onto its homology is equivalent to that of a Hodge decomposition, see \cite{chuang_laz_hodge,chuang_laz_feynman}. Such a SDR, therefore, always exists for any given finite-dimensional dg vector space.

\begin{definition}\label{def_SDR}
A SDR from $V$ to $U$ is a pair of even dg vector space morphisms $i:U\hookrightarrow V$ and $p:V\twoheadrightarrow U$ and an odd linear morphism $s:V\to V$ such that:
\begin{itemize}
\item $p\circ i=\id_U$;
\item $d\circ s+s\circ d=\id_V - i\circ p$;
\item the side conditions: $s\circ i=0$, $p\circ s=0$, and $s^2=0$.
\end{itemize}
\end{definition}

Generally, a SDR will be written as a triple $(i,p,s)$ when the super vector spaces are obvious from context.

The side conditions $si=0$, $ps=0$, and $s^2=0$ are included in the definition of a SDR, because they can be imposed at no cost as the following proposition states.

\begin{proposition}
Let $U,V$ be two dg vector spaces and $(i,p,s)$ be morphisms satisfying all the conditions of a SDR bar the side conditions, then $s$ can be replaced with a morphism $s'$ in such a way that the triple $(i,p,s')$ is a SDR.
\end{proposition}
\begin{proof}
If $s$ does not satisfy $s\circ i=0$ and $p\circ s=0$, then it can be replaced with $\tilde{s}=(d\circ s+s\circ d)\circ s\circ (d\circ s+s\circ d)$. The triple $(i,p,\tilde{s})$ now satisfies everything in Definition~\ref{def_SDR} with the possible exception of $\tilde{s}^2=0$. Replacing $\tilde{s}$ with $s^\prime = \tilde{s}\circ d\circ \tilde{s}$ means the triple $(i,p,s')$ is a SDR.
\end{proof}

\begin{proposition}
Let $V$ be a super vector space with an SDR $(i,p,s)$ onto its homology $\homology (V)$. Then $V=\im (i)\oplus\im (d)\oplus\im (s)$. Moreover, $d\colon \im(s)\to \im (d)$ is an odd linear isomorphism with inverse $s$.
\end{proposition}
\begin{proof}
The results follow directly from the properties of the SDR.
\end{proof}

\begin{definition}
Let $V=\im (i)\oplus\im (d)\oplus\im (s)$ be the decomposition of $V$ arising from a SDR onto its homology. Let $\langle,\rangle$ be the canonical symplectic pairing on $\im (d)\oplus\im (s)$ defined as follows: let $e_i$ be a basis for $\im (s)$ (so that $d e_i$ forms a basis for $\im (d)$) and set $\langle e_i , d e_j \rangle = \delta_{ij}$, with $\im (s)$ and $\im (d)$ isotropic subspaces.
\end{definition}

This definition is independent of the choice of basis.

The next two propositions are formal consequences of the definitions.

\begin{proposition}\label{prop_dual_sdr}
Let $U$ and $V$ be super vector spaces. Given a SDR $(i,p,s)$ from $V$ to $U$, a SDR from $V^\ast$ to $U^\ast$ is given by $(p^\ast, i^\ast ,s^\ast )$. \qed
\end{proposition}

\begin{proposition}\label{prop_sdr_compatible_with_sums}
Let $U_1$, $U_2$, $V_1$, and $V_2$ be super vector spaces. Suppose we have a SDR from $V_1$ to $U_1$ given by $(i_1 ,p_1,s_1)$ and another SDR from $V_2$ to $U_2$ given by $(i_2 , p_2 , s_2)$. A SDR from $V_1\oplus V_2$ to $U_1 \oplus U_2$ is given by $(i_1\oplus i_2 , p_1 \oplus p_2 , s_1 \oplus s_2)$. \qed
\end{proposition}

Clearly, Proposition~\ref{prop_sdr_compatible_with_sums} generalises to arbitrary direct sums.

\subsection{Integration}

Any finite-dimensional super vector space comes equipped with the canonical translation invariant volume form. Let $V$ be an exact super vector space (in the sense that $\homology (V)=0$) with canonical volume form $\mu$. Further, let $V=L_1\oplus L_2$ and equip it with the canonical symplectic pairing so that $L_1$ and $L_2$ are Lagrangian subspaces and $d\colon L_1 \to L_2$ is an isomorphism of odd degree.

\begin{proposition}
The quadratic form $\sigma \colon L_1 \to k$ given by $x\mapsto \langle x, dx \rangle$ is non-degenerate.
\end{proposition}
\begin{proof}
$\sigma$ is non-degenerate as a direct consequence of the definition of $\langle,\rangle$.
\end{proof}

\begin{proposition}
Let $f\in\hat{S}\Pi V^*$ be any function, then
\[
\int_{L_1} f e^{\frac{-\sigma}{2\hbar}} \mu
\]
is well-defined and can be given by the combinatorial sum over (possibly disconnected) stable graphs.
\end{proposition}
\begin{proof}
This is the standard Feynman diagrams approach, \cite{braun_maunder,costello} for example.
\end{proof}

The exact definition of a stable graph is not important for this paper. It is only important to know the restriction to `stable trees' (which is achieved by setting $\hbar=0$); we are restricting to trees whose vertices are at least tri-valent. More details are given in Section~\ref{sec_HTT}. The integral can also be written as an operator \cite{costello}.

\begin{proposition}\label{prop_BV_stokes}
Let $f\in\hat{S}\Pi V^*$ be any function, then
\[
\int_{L_1} \Delta \left(f e^{\frac{-\sigma}{2\hbar}}\right) \mu=0.
\]
\end{proposition}
\begin{proof}
This result is standard see \cite{braun_maunder,costello,schwarz}.
\end{proof}

\begin{definition}\label{def_integration}
Consider a super vector space $V$ (not necessarily exact) and a SDR onto its homology $\homology (V)$. Moreover, suppose that $V$ and $\homology (V)$ are equipped with odd symmetric forms and the SDR is compatible with these forms. The integral of $f\in \hat{S}\Pi V^\ast$ over $\im(s)$ is given by 
\[
\left(\id_{\homology (W)}\otimes \int_{\im(s)}\right) (f)\in \hat{S}\Pi \homology (V)^\ast.
\]
\end{definition}

\begin{proposition}
Consider a SDR $(i, p, s)$ of $V$ onto $\homology (V)$ as in Definition~\ref{def_integration}. Let $\iota_{\mathfrak{h}} \colon \mathfrak{h} [ \Pi \homology (V)] \to \mathfrak{h} [ \Pi V ]$ be the dgla morphism given by $f \mapsto ifp$. This morphism is a filtered quasi-isomorphism, and in particular induces a bijection of MC moduli sets on $\mathfrak{h} [ \Pi \homology (V)]$ and $\mathfrak{h} [ \Pi V ]$.
\end{proposition}
\begin{proof}
This is exactly Proposition~3.4 of \cite{braun_maunder}.
\end{proof}

The following theorem is the combination of several results of \cite{braun_maunder} and similar results can be found in \cite{doubek_jurco_pulmann}.

\begin{theorem}[Braun and Maunder]\label{thm_braun_maunder}
Consider a SDR $(i, p, s)$ of $V$ onto $\homology (V)$ as in Definition~\ref{def_integration}. The morphism $\rho \colon \MC \left(\mathfrak{h}[\Pi V]\right) \to \MC \left(\mathfrak{h}[\Pi \homology (V)]\right)$ given by
\[
\rho (m) = \hbar \log \left(\int_{L_1} e^{\frac{m}{\hbar}} e^{\frac{-\sigma}{2\hbar}} \mu\right)
\]
is well-defined and can presented as the combinatorial sum over connected stable graphs
\[
\sum_{G} \frac{F(G)}{|\Aut(G)|},
\]
where $F(G)$ is a function belonging to $\hat{S}\Pi \homology (V)^* [|\hbar|]$ constructed from the stable graph $G$ and $|\Aut(G)|$ denotes the order of the automorphism group of the stable graph $G$.

Moreover, $\iota_{\mathfrak{h}} \circ \rho (m)$ is equivalent to $m$ (as MC elements in $\mathfrak{h}[\Pi V]$). \qed
\end{theorem}

In other words, Theorem~\ref{thm_braun_maunder} states that the integral of a quantum \Linf-algebra structure gives the minimal model of that structure. We wish to make use of this theorem in transferring non-quantum structures using the doubling constructions.

In the proof of Theorem~\ref{thm_braun_maunder}, the following fact is used (Proposition~2.19 of \cite{braun_maunder}). This fact is also useful in this paper.

\begin{proposition}\label{prop_integration_respects_homotopy}
Let $V$ be a super vector space equipped with a SDR onto its homology $\homology (V)$. Given a homotopy $h(t)$ of quantum \Linf-algebra structures $m_1$ and $m_2$ on $V$, the function
\[
\tilde{h}(t) = \hbar \log \left( \int_{L_1} e^{\frac{h(t)}{\hbar}} e^{\frac{-\sigma}{2\hbar}} \mu\right)
\]
defines a homotopy of $\rho (m_1)$ and $\rho (m_2)$. (Note that integration has been extended $t,dt$-linearly.) \qed
\end{proposition}

\subsection{Integrating doubles}

Let $V$ be a super vector space with a SDR $(i,p,s)$ onto its homology. Note that we do not assume that $V$ is equipped with any bilinear form, because we intend to double $V$. Proposition~\ref{prop_dual_sdr} and Proposition~\ref{prop_sdr_compatible_with_sums} allow us to combine the SDR of $V$ onto $\homology (V)$ with the SDR of $\Pi V^\ast$ onto $\homology (\Pi V^\ast)$ given by $(\Pi p^\ast,\Pi i^\ast,\Pi s^\ast)$ to obtain the Hodge decomposition of the direct sum $V\oplus \Pi V^\ast$ given by
\[
V\oplus \Pi V^\ast = (\homology (V) \oplus \homology (\Pi V^\ast)) \oplus (L_s \oplus \Pi D_s^\ast) \oplus (D_s \oplus \Pi L_s^\ast),
\]
such that $\overline{d}=d\oplus \Pi d^\ast \colon L_s \oplus \Pi D_s^\ast \to D_s \oplus \Pi L_s^\ast$ is an isomorphism with inverse $\overline{s}=s\oplus \Pi s^\ast$. Let $\mathcal{L}_{od,s}$ denote the direct sum $L_s \oplus \Pi D_s^\ast$.

\begin{proposition}
The quadratic form $\sigma_{od} \colon \mathcal{L}_{od,s} \to k$ given by $x\mapsto \langle x,\overline{d} \xi_x \rangle$ is non-degenerate.
\end{proposition}
\begin{proof}
The statement follows from the definitions of $\overline{d}$ and $\langle,\rangle$.
\end{proof}

Therefore, when integrating the odd double we integrate over $\mathcal{L}_{od,s}$ and against the `Gaussian measure' corresponding to the non-degenerate quadratic form $\sigma_{od}$.

\section{Minimal models of unimodular L-infinity structures}\label{sec_minimal_unimodular}

We will now look at how integrating allows us to transfer a unimodular \Linf-algebra structure on a general super vector space $V$ to an equivalent unimodular \Linf-algebra structure on its homology $\homology (V)$ (its minimal model).

\begin{definition}
Let $V$ be a super vector space with a split quasi-isomorphism $\homology (V) \to V$. Let \[
\iota_{\mathfrak{g}} \colon \mathfrak{g}[\Pi \homology (V)] \to \mathfrak{g}[\Pi V]
\]
be the induced filtered quasi-isomorphism of dglas. Given a unimodular \Linf-algebra structure $(m,f)$ on $V$, the minimal model of $(m,f)$ is a unimodular \Linf-algebra structure $(m^\prime,f^\prime)$ on $\homology (V)$ such that $\iota_{\mathfrak{g}} (m^\prime ,f^\prime)$ and $(m,f)$ are equivalent unimodular \Linf-algebra structures on $V$.
\end{definition}

\begin{remark}
The above definition in the context of \Linf-algebras is equivalent to the usual definition of the minimal model. That is, the minimal of an \Linf-algebra $(V,m)$ is a \Linf-algebra $(\homology (V),m^\prime)$ that is \Linf-quasi-isomorphic to $V$.
\end{remark}

The existence and uniqueness for minimal models of unimodular \Linf-algebras has already been established in Proposition~\ref{prop_unimodular_structure_bijection}. However, the proof of this proposition is not constructive. The goal of this section is to give an explicit construction for the minimal model of a unimodular \Linf-algebra structure.

\begin{definition}
Let $(m,f)$ be a unimodular \Linf-algebra structure on $V$. Further suppose $V$ comes equipped with a SDR onto its homology $\homology (V)$. Define $T_1$ to be the map $\MC(\mathfrak{g}[\Pi V])\to \MC (\mathfrak{g}[\Pi \homology (V)])$ given by the composition
\[
T_1 (m,f)=\Hodq \circ \rho \circ \Dodq (m,f).
\]
\end{definition}

\begin{theorem}\label{thm_minimal_model_for_unimodular}
Let $V$ be a super vector space with a SDR $(i,p,s)$ onto its homology $\homology (V)$. If $(m,f)$ defines a unimodular \Linf-algebra structure on $V$, then $T_1(m,f)$ defines the minimal model of $(m,f)$.
\end{theorem}
\begin{proof}
First, we know that the unimodular \Linf-algebra structure $(m,f)$ defines a quantum \Linf-algebra structure $\Dodq (m,f)$, by Proposition~\ref{prop_unimodular_to_quantum}. Taking the integral of this structure results in a quantum \Linf-algebra structure $\widetilde{m}_q \in \mathfrak{h}[\Pi (\homology (V)^\ast \oplus \Pi \homology(V))]$ that is equivalent to $\Dodq (m,f)$. Now, denoting the constant part in $\hbar$ of $\widetilde{m}_q$ by $\widetilde{m}$ and the linear part by $\widetilde{f}$. The pair $(X_{\widetilde{m}},\widetilde{f})$ define a unimodular \Linf-algebra structure on $\homology (V)^\ast \oplus \Pi \homology (V)$. Moreover, the pair $(X_{\widetilde{m}},\widetilde{f})$ live in the correct Lie subalgebra so that halving this structure results in a unimodular \Linf-algebra structure on $\homology (V)$, by Proposition~\ref{prop_havling_quantum_to_unimodular}. Therefore, the map $T_1$ is well-defined.

It is straightforward to see that $T_1 \circ \iota_{\mathfrak{g}}$ restricted to the set $\MC(\mathfrak{g}[\Pi \homology (V)])$ is the identity, and so $T_1$ induces the inverse bijection to $\iota_{\mathfrak{g}}$ on MC moduli sets. Therefore $\iota_{\mathfrak{g}} \circ T_1 (m,f)$ is homotopic to $(m,f)$.
\end{proof}

\section{Minimal models of L-infinity structures}\label{sec_minimal_linf}

Similar to the previous section, we will now look at how integrating allows us to transfer an \Linf-algebra structure on a general super vector space $V$ to a \Linf-quasi-isomorphic \Linf-algebra structure on its homology $\homology (V)$. Minimal models of \Linf-algebra structures are already understood, see \cite{kontsevich} for example.

\begin{definition}
Let $(V,m)$ be an \Linf-algebra and suppose there is a SDR of $V$ onto its homology $\homology (V)$. Let $T_2 \colon \MC (\Der_{\geq 2} (\hat{S}\Pi V^\ast) )\to \MC (\Der_{\geq 2} (\hat{S}\Pi \homology (V)^\ast))$ be the map given by the composition
\[
T_2 (m) = \Hev \circ \Hod \circ |_{\hbar= 0} \circ \rho \circ \Dodq \circ \Dev (m).
\]
\end{definition}

\begin{theorem}\label{thm_minimal_model_for_general}
Let $V$ be a super vector space with a SDR $(i,p,s)$ onto its homology $\homology (V)$. If $m$ defines an \Linf-algebra structure on $V$, then $T_2 (m)$ defines the minimal model of $m$.
\end{theorem}
\begin{proof}
Proposition~\ref{prop_even_double_is_unimodular} shows that $\Dev (m)$ is strictly unimodular \Linf-algebra structure and so $\Dodq \circ \Dev (m)$ is a well-defined quantum \Linf-algebra structure. Therefore, $\rho (\Dodq \circ \Dev (m))$ defines a quantum \Linf-algebra equivalent to $\Dodq \circ \Dev (m)$ and setting $\hbar=0$ defines an \Linf-algebra structure on the `quadrupled' $\homology (V)$. Moreover, this \Linf-algebra lives in the correct Lie subalgebra to allow for halving. Halving twice, we obtain the \Linf-algebra $T_2 (m)$ on $\homology (V)$, meaning the map is well-defined.

The SDR induces a filtered quasi-isomorphism
\[
\iota \colon \Der_{\geq 2} \left(\hat{S}\Pi \homology (V)^\ast \right) \to \Der_{\geq 2} \left(\hat{S}\Pi V^\ast \right)
\]
(and hence a bijection of MC moduli sets). Moreover, the restriction of $T_2 \circ \iota$ to $\MC (\Der_{\geq 2} (\hat{S}\Pi \homology (V)^\ast))$ is the identity morphism. Therefore, the map $\iota \circ T_2 (m)$ is homotopic to $m$.
\end{proof}

\section{Homotopy transfer theorem}\label{sec_HTT}

Adding a slight extension to Theorem~\ref{thm_minimal_model_for_general} and changing the formulation slightly, we recover the Homotopy Transfer Theorem for \Linf-algebras (see \cite{loday_vallette,vallette_alg_htpy_operad}, for example). In the following sections, we will explain how the sums over Feynman diagrams arising in the BV-formalism are related to the sums over rooted trees in the Homotopy Transfer Theorem (HTT).

We note that we are unable to state a similar theorem for unimodular or quantum \Linf-algebra structures, because there is no notion of unimodular or quantum \Linf-morphism, let alone quasi-isomorphism.

\subsection{A graphical calculus}

In order to display how $T_2$ recovers the HTT by recovering the known formulas for transferred structures, i.e.~those in terms of trees, we must first look at how doubling and halving act on graphs. In the process of doing this we make use of some parts of a graphical calculus for symmetric monoidal categories. The graphical calculus used in this paper is close to those commonly used in the literature (see \cite{fiorenza_murri,selinger_graphical_languages} for example).

Objects of the category are thought of as directed `strings' or `edges' and morphisms of the category as `nodes' or `vertices' of the diagrams. A given diagram will be read from top to bottom with the inputs all along a horizontal line at the very top of the diagram and the outputs placed all along a horizontal at the very bottom of the diagram; for example, a morphism $f\colon U\otimes V\to W$ is represented by the decorated directed graph
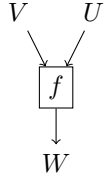
\begin{center}
\begin{tikzpicture}
\node (in1) at (-0.5,2) {$V$};
\node (in2) at (0.5,2) {$U$};
\node[draw] (f) at (0,1) {$f$};
\node (out) at (0,0) {$W$};

\draw[->] (in1) -- (f);
\draw[->] (in2) -- (f);
\draw[->] (f) -- (out);
\end{tikzpicture}
\end{center}
The category of super vector spaces is symmetric monoidal, and, therefore, the order of the input edges is unimportant (as oppose to the braided case, for instance, where it is necessary to make note of under and over crossings). To represent the dual of an object, the orientation of an edge is reversed. The absence of a string represents the unit of the monoidal product, i.e.~the ground field. Therefore, a diagram with no inputs (resp.~no outputs) represents a function with domain (resp.~codomain) given by the ground field.

A general diagram is usually made up of compositions of ones like the one above, composed in the obvious way to make a larger directed graph.

\subsection{Doubling and halving in the graphical calculus}

The known formulas for the HTT make use of \Linf-algebra structures on $V$ given in the form of a collection of linear maps $(\Pi V)^{\otimes n} \to \Pi V$. Thus, we will first re-imagine the action of $T_2$ (and more specifically the doubling and halving maps) in this guise.

We will first look at the action of the doubling and halving on flowers. The even double takes a symmetric multi-linear map
\[
m_n \colon (\Pi V)^{\otimes n} \to \Pi V
\]
to an element of $(\Pi V^*)^{\otimes n} \otimes \Pi V$ which can be thought of as a map
\[
\Dev (m_n) \colon k \to (\Pi V^*)^{\otimes n} \otimes \Pi V \left(\subset \hat{S}\Pi V^\ast \otimes \Pi V \subset \hat{S}\Pi (V \otimes V)^\ast \right).
\]
On a symmetric multi-linear map, the even double acts via the tensor-hom adjunction. Pictorially, all the even double is doing is `flipping' the $n$ inputs:
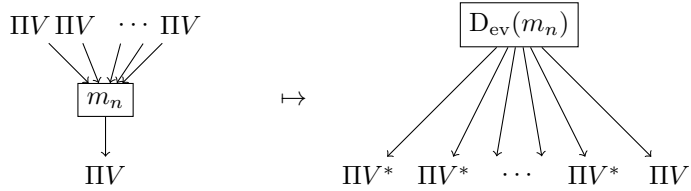
\begin{center}
\begin{tikzpicture}
\node (out1) at (0,0) {$\Pi V$};
\node[draw] (fn1) at (0,1) {$m_n$};

\node (in1) at (-1,2) {$\Pi V$};
\node (in2) at (-0.4,2) {$\Pi V$};
\node at (0.4,2) {$\cdots$};
\node (in3) at (1,2) {$\Pi V$};

\draw[->] (in1) -- (fn1);
\draw[->] (in2) -- (fn1);
\draw[->] (in3) -- (fn1);
\draw[->] (0.2,1.75) -- (fn1);
\draw[->] (0.5,1.75) -- (fn1);
\draw[->] (fn1) -- (out1);

\node at (2.5,1) {$\mapsto$};

\node (out2) at (7.5,0) {$\Pi V$};
\node[draw] (fn2) at (5.5,2) {$\Dev(m_n)$};

\node (in4) at (3.5,0) {$\Pi V^\ast$};
\node (in5) at (4.5,0) {$\Pi V^\ast$};
\node at (5.5,0) {$\cdots$};
\node (in6) at (6.5,0) {$\Pi V^\ast$};

\draw[<-] (in4) -- (fn2);
\draw[<-] (in5) -- (fn2);
\draw[<-] (in6) -- (fn2);
\draw[<-] (5.2,0.25) -- (fn2);
\draw[<-] (5.8,0.25) -- (fn2);
\draw[->] (fn2) -- (out2);
\end{tikzpicture}
\end{center}
The odd double applied to a multi-linear map is similar, only there is a parity change of the output.

To describe the action of the odd double of the above $\Dev (m_n)$, assume that $\Dev (m_n)=f_1 f_2 \cdots f_n v\in \hat{S}\Pi V^\ast \otimes \Pi V$, where $f_1,f_2,\cdots, f_n \in \Pi V^\ast$ and $v\in \Pi V$. Moreover, let $W=V\oplus V^\ast$ to ease notation. After becoming a derivation and doubling, $\Dev(m_n)$ becomes the sum of $n+1$ terms:
\[
(\pm ) f_1 f_2 \cdots f_n \xi_v + (\pm) f_1 f_2 \cdots f_{n-1} \xi_{f_n} v + \cdots + (\pm) \xi_{f_1} f_2 \cdots f_n v \in \hat{S}\Pi W^\ast \otimes W,
\]
where $\xi_x\in \Pi W$ is the dual element of $x\in W$. So, taking the odd double of the even double takes a single flower and converts it to a sum represented by the flower:
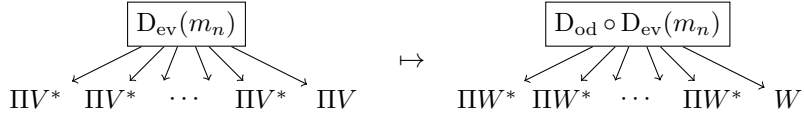
\begin{center}
\begin{tikzpicture}
\node (out2) at (1.5,0) {$\Pi V$};
\node[draw] (fn2) at (-0.5,1) {$\Dev(m_n)$};

\node (in4) at (-2.5,0) {$\Pi V^\ast$};
\node (in5) at (-1.5,0) {$\Pi V^\ast$};
\node at (-0.5,0) {$\cdots$};
\node (in6) at (0.5,0) {$\Pi V^\ast$};

\draw[<-] (in4) -- (fn2);
\draw[<-] (in5) -- (fn2);
\draw[<-] (in6) -- (fn2);
\draw[<-] (-0.8,0.25) -- (fn2);
\draw[<-] (-0.2,0.25) -- (fn2);
\draw[->] (fn2) -- (out2);

\node at (2.5,0.5) {$\mapsto$};

\node[draw] (fn3) at (5.5,1) {$\Dod\circ\Dev(m_n)$};

\node (in7) at (3.5,0) {$\Pi W^\ast$};
\node (in8) at (4.5,0) {$\Pi W^\ast$};
\node at (5.5,0) {$\cdots$};
\node (in10) at (6.5,0) {$\Pi W^\ast$};
\node (in11) at (7.5,0) {$W$};

\draw[<-] (in7) -- (fn3);
\draw[<-] (in8) -- (fn3);
\draw[<-] (5.2,0.25) -- (fn3);
\draw[<-] (5.8,0.25) -- (fn3);
\draw[<-] (in10) -- (fn3);
\draw[<-] (in11) -- (fn3);
\end{tikzpicture}
\end{center}
The halving maps are now the obvious reverse of the above.

We calculate the amplitudes $F(G)$ of Feynman diagrams by contracting edges using the the propagator $?\mapsto \langle ?, s(?)\rangle=\langle,\rangle \circ (\id\otimes s)$. After halving, the propagator becomes $s_W \colon W \to W$. For suitable graphs, we have
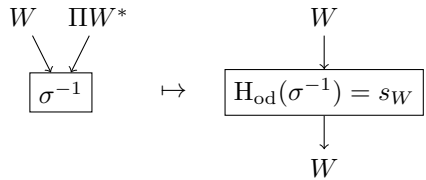
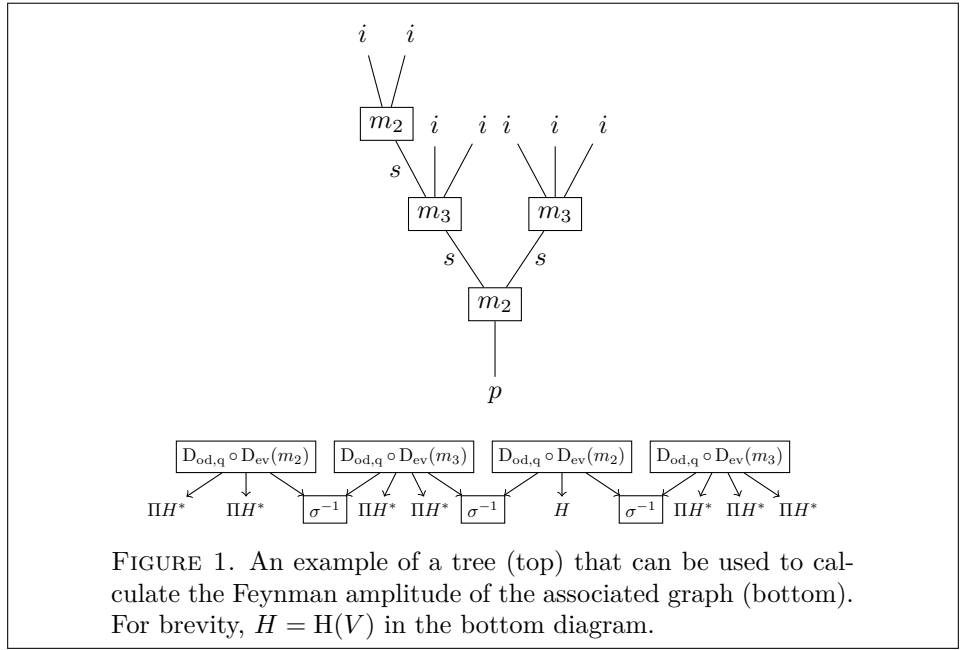
\begin{center}
\begin{tikzpicture}
\node[draw] (prop) at (0,1) {$\sigma^{-1}$};
\node (out1) at (-0.5,2) {$W$};
\node (out2) at (0.5,2) {$\Pi W^\ast$};

\draw[<-] (prop) -- (out1);
\draw[<-] (prop) -- (out2);

\node at (1.5,1) {$\mapsto$};

\node[draw] (Hprop) at (3.5,1) {$\Hod(\sigma^{-1})=s_W$};
\node (in1) at (3.5,2) {$W$};
\node (out3) at (3.5,0) {$W$};

\draw[->] (in1) -- (Hprop);
\draw[->] (Hprop) -- (out3);
\end{tikzpicture}
\end{center}
Note that for elements of $W^{\otimes 2}$ the contraction is zero --- likewise, for $(\Pi W^\ast)^{\otimes 2}$.

Let us now look at the action of the map $T_2$ on those graphs that arise when integrating solutions to the QME. Since the map $T_2$ sets $\hbar$ to zero, the only graphs that remain are rooted trees with at least tri-valent vertices. The tree-level truncation of the Feynman diagram expansion of the integral
\[
\hbar \log \left( \int_{L_1} e^{\frac{1}{\hbar}\Dodq \circ \Dev (m)} e^{-\frac{\sigma}{2\hbar}} \right)
\]
consists of trees with vertices that are at least tri-valent and decorated by the components of
\[
\Dodq \circ \Dev (m)|_{\homology (V) \oplus \Pi \homology (V)^\ast}.
\]

\begin{proposition}
For a connected at least tri-valent graph, $\Hod \circ F (G)$ is calculated by placing $i$ at the ends of the `incoming legs', $p$ at the `outgoing legs', and $s$ at any internal edge. \qed
\end{proposition}

\begin{proposition}
For a connected at least tri-valent graph, $\Hev\circ \Hod \circ F (G)$ is calculated by placing $i$ at the ends of the `incoming legs', $p$ at the `outgoing legs', and $s$ at any internal edge. \qed
\end{proposition}

For example, the tree of Figure~\ref{fig_example_tree} can be used to calculate the Feynman amplitude of the corresponding graph, also in Figure~\ref{fig_example_tree}.

\begin{figure}
\centering
\begin{tikzpicture}[grow=up,scale=0.8,level 1/.style={sibling distance=30mm},level 2/.style={sibling distance=20mm},level 3/.style={sibling distance=8mm}]
\node {$p$}
child {
	node[draw] {$m_2$}
	child {
		node[draw] {$m_3$}        
			child {
				node[circle] {$i$}
				edge from parent
			}
			child {
				node[circle] {$i$}
				edge from parent
			}
			child {
				node[circle] {$i$}
				edge from parent
			}
			edge from parent 
			node[right] {$s$}
	}
    child {
		node[draw] {$m_3$}        
			child {
				node[circle] {$i$}
				edge from parent
			}
			child {
				node[circle] {$i$}
				edge from parent
			}
			child {
				node[draw] {$m_2$}
					child {
						node[circle] {$i$}
						edge from parent
						}
					child {
						node[circle] {$i$}
						edge from parent
						}
				edge from parent
				node[left,yshift=-0.1em] {$s$}
			}
			edge from parent 
			node[left] {$s$}
	}
edge from parent 
};
\end{tikzpicture}

\vspace*{1em}

\begin{tikzpicture}[scale=0.7, every node/.style={scale=0.7}]
\node[draw] (m21) at (0,0) {$\Dodq \circ \Dev (m_2)$};
\node[draw] (m31) at (3,0) {$\Dodq \circ \Dev (m_3)$};
\node[draw] (m22) at (6,0) {$\Dodq \circ \Dev (m_2)$};
\node[draw] (m32) at (9,0) {$\Dodq \circ \Dev (m_3)$};

\node[draw] (prop1) at (1.5,-1) {$\sigma^{-1}$};
\node[draw] (prop2) at (4.5,-1) {$\sigma^{-1}$};
\node[draw] (prop3) at (7.5,-1) {$\sigma^{-1}$};

\node (in1) at (-1.5,-1) {$\Pi H^\ast$};
\node (in2) at (0,-1) {$\Pi H^\ast$};

\node (in3) at (2.5,-1) {$\Pi H^\ast$};
\node (in4) at (3.5,-1) {$\Pi H^\ast$};

\node (in5) at (6,-1) {$H$};

\node (in6) at (8.5,-1) {$\Pi H^\ast$};
\node (in7) at (9.5,-1) {$\Pi H^\ast$};
\node (in8) at (10.5,-1) {$\Pi H^\ast$};

\draw[->] (m21) -- (in1);
\draw[->] (m21) -- (in2);
\draw[->] (m21) -- (prop1);

\draw[->] (m31) -- (prop1);
\draw[->] (m31) -- (in3);
\draw[->] (m31) -- (in4);
\draw[->] (m31) -- (prop2);

\draw[->] (m22) -- (prop2);
\draw[->] (m22) -- (in5);
\draw[->] (m22) -- (prop3);

\draw[->] (m32) -- (prop3);
\draw[->] (m32) -- (in6);
\draw[->] (m32) -- (in7);
\draw[->] (m32) -- (in8);
\end{tikzpicture}
\caption{An example of a tree (top) that can be used to calculate the Feynman amplitude of the associated graph (bottom). For brevity, $H=\homology (V)$ in the bottom diagram.}
\label{fig_example_tree}
\end{figure}

All told, restating Theorem~\ref{thm_minimal_model_for_general} using the above graphical calculus we arrive at the following familiar presentation of the HTT for \Linf-algebras.

\begin{theorem}\label{thm_HTT_via_multi-linear}
Let $\lbrace m_n \colon (\Pi V)^{\otimes n} \to \Pi V \rbrace_{n\geq 2}$ be an \Linf-algebra structure on a super vector space $V$, then a SDR $(i,p,s)$ of $V$ onto its homology $\homology (V)$ induces the \Linf-algebra structure $\lbrace l_n \colon (\Pi \homology (V))^{\otimes n} \to \Pi \homology (V) \rbrace_{n\geq 2}$ on $\homology (V)$ given by
\begin{center}
\begin{tikzpicture}
\node (out1) at (1,0) {};
\node[draw] (fn1) at (1,1) {$l_n$};

\node (in1) at (0,2) {$1$};
\node (in2) at (0.5,2) {$2$};
\node (in3) at (1,1.9) {};
\node at (1.25,2) {$\cdots$};
\node (in4) at (1.5,1.9) {};
\node (in5) at (2,2) {$n$};

\draw[-] (in1) -- (fn1);
\draw[-] (in2) -- (fn1);
\draw[-] (in3) -- (fn1);
\draw[-] (in4) -- (fn1);
\draw[-] (in5) -- (fn1);
\draw[-] (fn1) -- (out1);

\node at (2.5,1) {$=$};

\node at (3.5,1) {$\displaystyle\sum_{\operatorname{Tree}_n} \pm$};

\node (out2) at (5.5,0) {};
\node[circle,fill,scale=0.1] (m1) at (4.4,1.5) {};
\node[circle,fill,scale=0.1] (m2) at (5,1) {};
\node[circle,fill,scale=0.1] (m3) at (5.5,0.5) {};
\node[circle,fill,scale=0.1] (m4) at (6,1) {};

\node (in10) at (4,2) {};
\node (in11) at (4.2,2) {};
\node (in12) at (4.4,2) {};
\node (in13) at (4.6,2) {};
\node (in14) at (4.8,2) {};

\node (in15) at (5,1.5) {};
\node (in16) at (5.4,1.5) {};

\node (in17) at (5.7,1.5) {};
\node (in18) at (5.9,1.5) {};
\node (in19) at (6.1,1.5) {};
\node (in20) at (6.3,1.5) {};

\draw[-] (in10) -- (m1);
\draw[-] (in11) -- (m1);
\draw[-] (in12) -- (m1);
\draw[-] (in13) -- (m1);
\draw[-] (in14) -- (m1);

\draw[-] (m1) -- (m2);
\draw[-] (in15) -- (m2);
\draw[-] (in16) -- (m2);

\draw[-] (in17) -- (m4);
\draw[-] (in18) -- (m4);
\draw[-] (in19) -- (m4);
\draw[-] (in20) -- (m4);

\draw[-] (m2) -- (m3);
\draw[-] (m4) -- (m3);
\draw[-] (m3) -- (out2);
\end{tikzpicture}
\end{center}
where $\operatorname{Tree}_n$ denotes the set of rooted trees with $n$ leaves and the labelling of the tree on the right hand side is given as follows: the inputs edges are labelled by $i$, the internal edges are labelled by $s$, the output edge are labelled by $p$, and the vertices are labelled by the operations $\lbrace m_n \rbrace_{n\geq 2}$ of the correct arity.

Moreover, the quasi-isomorphism $i$ induces an \Linf-quasi-isomorphism of the above \Linf-algebra structures. \qed
\end{theorem}

As mentioned previously, there seems to be no good notion of a unimodular or quantum \Linf-algebra ((quasi)-iso)morphism and so the constructions of minimal models in those cases are not an instance of the HTT (because it is unclear what it would mean for the morphism $i$ to extend to an $\infty$-morphism in this case). However, it would be nice to find a more general operad for which the BV-formalism is homotopy transfer that makes full use of the `quantum' structure available.

\section*{Acknowledgements}

I would like to thank Andrey Lazarev for helpful discussions. I would also like to thank Christopher Braun, whose collaboration preceded this work. 

\bibliographystyle{plain}
\bibliography{my_bib}

\end{document}